\documentclass[a4paper,12pt,english,reqno]{amsart}
\usepackage{amsmath,amssymb,amsfonts,dsfont,amsthm,upgreek,bm,mathtools}    
\usepackage[utf8]{inputenc}
\usepackage[english]{babel}
\usepackage[T1]{fontenc} 
\usepackage{graphicx,xcolor}
\usepackage{float}
\usepackage[font=small,labelfont=bf]{caption}
\usepackage[a4paper, margin=3cm]{geometry}
\usepackage{tikz,pgfplots}
\usepackage{tkz-fct}
\usepgfplotslibrary{polar}

\usepackage{enumerate}
\usepackage{scalerel,stackengine}
\stackMath
\newcommand\reallywidecheck[1]{%
\savestack{\tmpbox}{\stretchto{%
  \scaleto{%
    \scalerel*[\widthof{\ensuremath{#1}}]{\kern-.6pt\bigwedge\kern-.6pt}%
    {\rule[-\textheight/2]{1ex}{\textheight}}
  }{\textheight}%
}{0.6ex}}%
\stackon[1pt]{#1}{\scalebox{-0.8}{\tmpbox}}%
}

\newcommand{\e}{\mathrm{e}}

\newcommand{\Op}{\mathrm{Op}}

\newcommand{\supp}{\mathrm{supp\,}}

\newcommand{\Ima}{\mathrm{Im\,}}
\newcommand{\Rea}{\mathrm{Re\,}}

\newcommand{\mO}{\mathcal{O}}
\newcommand{\C}{\mathds{C}}
\newcommand{\R}{\mathds{R}}

\newcommand{\N}{\mathds{N}}
\newcommand{\Z}{\mathds{Z}}

\renewcommand{\ge}{\geqslant}
\renewcommand{\geq}{\geqslant}
\renewcommand{\le}{\leqslant}
\renewcommand{\leq}{\leqslant}

\newtheorem{thm}{Theorem}

\newtheorem{prop}[thm]{Proposition}

\newtheorem{rem}[thm]{Remark}
\theoremstyle{remark}

\numberwithin{equation}{section}
\numberwithin{thm}{section}

\usepackage{hyperref}
\hypersetup{pdfborder=0 0 0, 
	    colorlinks=true,
	    citecolor=blue,
	    linkcolor=blue,
	    urlcolor=blue,
	    pdfauthor={Johannes Sjoetrand and Martin Vogel}
	   }
\title{Tunneling for the $\overline{\partial}$-operator}
\author{Johannes Sj\"ostrand}
\address[Johannes Sj\"ostrand]{IMB, 
  Universit\'e de Bourgogne Franche-Comt\'e, 
  UMR 5584 du CNRS, 
  9, avenue Alain Savary - BP 47870 - 21078 Dijon Cedex, France.}
\email{johannes.sjostrand@u-bourgogne.fr}
\author{Martin Vogel}
\address[Martin Vogel]{Institut de Recherche Math{\'e}matique Avanc{\'e}e 
- UMR 7501, CNRS et Universit{\'e} de Strasbourg, 7 rue René-Descartes, 
67084 Strasbourg Cedex, France.}
\email{vogel@math.unistra.fr}
\date{\today}
\subjclass[2010]{}
\begin{document}
\begin{abstract} 
We study the small singular values of the $2$-dimensional semiclassical 
differential operator $P = 2\,\e^{-\phi/h}\circ hD_{\overline{z}}\circ \e^{\phi/h}$ 
on $S^1+iS^1$ and on $S^1+i\R$ where $\phi$ is given by $\sin y$ and by $y^3/3$, 
respectively. The key feature of this model is the fact that we can pinpoint 
precisely where in phase space the Poisson bracket $\{p,\overline{p}\}=0$, 
where $p$ is the semiclassical symbol of $P$. 
\par 
We give a precise asymptotic description of the exponentially small singular 
values of $P$ by studying the tunneling effects of an associated Witten complex. 
We use these asymptotics to determine a Weyl law for the exponentially small 
singular values of $P$. 
\end{abstract}
 \maketitle
\setcounter{tocdepth}{1}
\tableofcontents
\section{Introduction}\label{int}
\textbf{The setting.} It is well known that the norm of the resolvent of 
linear non-selfadjoint operators $P$ on complex Hilbert space can be large 
even far away from the spectrum of the operator. As a consequence the spectrum 
of $P$ can be unstable under small perturbations. This phenomenon 
-- also called pseudo-spectral phenomenon -- was initially considered a drawback 
as it can be at the source of computational errors in numerical 
analysis \cite{RT92,Tr97,TrEm05}. However, it is also at the source of new and 
interesting results. A large resolvent norm of $P-z$, $z\in\C$, is related to  
the existence of approximate solutions to the eigenvector problem $(P-z)u=0$, 
which is in turn equivalent to the existence of small singular values of $P-z$. 
\par 
In this paper we study the small singular values of a 
non-selfadjoint semiclassical differential operator $P_h$, of Fredholm 
index $0$, on $X$ being either $\R^d$ or a $d$-dimensional smooth compact manifold, 
with semiclassical parameter $h>0$ and symbol $p$. A natural condition for the 
non-selfadjointness of $P_h$ is that $\{p,\overline{p}\}\not\equiv 0$ on 
$T^*X$-- here $\{\cdot,\cdot\}$ denotes the Poisson bracket -- since this 
is the principal symbol of $\frac{i}{h}[P_h,P_h^*]$. It was shown in the 
string of works \cite{Da97,Da99,Da99b,Zw01,DSZ04} that if there is a point 
$\rho_+ \in T^*X\cap p^{-1}(z)$, $z\in\C$, such that 
$\frac{1}{i}\{p,\overline{p}\}(\rho_+)>0$ then there exists a smooth function 
$e_+$ on $X$, mircolocalized to $\rho_+$, such that 
\begin{equation}\label{intro_eq1}
	(P-z)e_+ = O(h^\infty). 
\end{equation}
\par 
The existence of such a quasimode implies the existence of a small singular 
value of $(P-z)$ of size $O(h^\infty)$. Moreover, when $p$ is analytic, then 
we may replace the above error estimates by $O(\e^{-1/Ch})$.
\par 
On the other hand, near regions where $\frac{1}{i}\{p,\overline{p}\}<0$ on 
$p^{-1}(z)$ we have that microlocally $|P-z|\geq h/C$, where we set 
$|P-z|:=((P-z)^*(P-z))^{1/2}$. This indicates that the eigenfunctions of 
$|P-z|$, i.e. the singular vectors of $P-z$ should be mircolocalized near 
the regions in $p^{-1}(z)$ where $\frac{1}{i}\{p,\overline{p}\}>0$. 
Similar observations, with a factor of $(-1)$ in front of the Poisson 
bracket, hold for the singular vectors of $(P-z)^*$. 
%
%
%
\par
In dimension $d=1$, when $p^{-1}(z)$ is compact and $\{p,\overline{p}\}\neq 0$ 
on $p^{-1}(z)$, we can show that that there are $|p^{-1}(z)|=:n_0$ many singular 
values of $(P-z)$ and $(P-z)^*$ which are of size $O(h^\infty)$ (or 
$O(\e^{-1/Ch})$ when $p$ is analytic). However, the $n_0+1$-st singular value will 
typically be of size $\sqrt{h}$. 
\par 
In higher dimensions the situation is more complicated. We will focus on the 
situation where $\{p,\overline{p}\}\not\equiv 0$ and separated by regions in 
$p^{-1}(z)$ where $\{p,\overline{p}\} = 0$ but $\{p,\{p,\overline{p}\}\} \neq  0$. 
For related results in more degenerate cases, see \cite{Pr03,Pr06,Pr08}.
\\
\par 
The aim of this paper is to study the number of small singular values for certain  
$2$-dimensional model operators $P_h$ with the symbols $p$ such that 
$\{p,\overline{p}\} \neq 0$ or $\{p,\{p,\overline{p}\}\} \neq 0$.
\\
\par
\textbf{The model.} We will work on the complex 1-dimensional manifold
$X+i Y$, where $X=S^1=\R/2\pi \Z$ and $Y$ is equal to $\R$ or $S^1$. 
We frequently identify $X+iY$ with the real 2 dimensional manifold $X\times Y$, 
by writing $z=x+iy$, $x\in X$, $y\in Y$.
\par
%
Let $h\in ]0,h_0]$, $h_0>0$, and let $\phi(z)$ be a real-valued smooth function 
on $X+iY$ and consider the operator 
\begin{equation}\label{eq1}
	P = 2\,\e^{-\phi/h}\circ hD_{\overline{z}}\circ \e^{\phi/h} 
	  = 2(hD_{\overline{z}} + (D_{\overline{z}}\phi))
	  =hD_x+h\partial_y + \partial_y\phi,
\end{equation}
where $\partial_{\overline{z}} = \frac{1}{2}( \partial_x + i \partial_y)$ 
and $D_{\overline{z}} = \frac{1}{i}\partial_{\overline{z}} = 
\frac{1}{2}( D_x + i D_y)$. With a slight abuse of notation we write 
$\phi(z) = \phi(y)$, for $z=x+iy\in X+iY$, and we choose the following 
two model cases
\begin{equation}\label{eq6a}
\phi(y) =
	\begin{cases} 
		\frac{1}{3} y^3, \text{ when } Y=\mathbb{R},\\ 
		\sin y, \text{ when } Y= S^1.
	\end{cases}
\end{equation}
We equip $P$ with the domain 
\begin{equation}\label{eq6a.1.1}
\mathcal{D}(P)
=
\begin{cases}
	\{ u\in
	L^2(S^1_x\times Y);\, (1+y^2)u,\,
				 hD_xu,\,hD_yu\in L^2(S^1_x\times Y) \}, 
	\quad Y=\mathbb{R}_y,
	\\
	H_h^1(S^1_x\times Y)
	=\{u\in L^2(S^1_x\times Y);\,hD_xu,\,hD_yu\in L^2(S^1_x\times Y) \},
	\quad Y=S^1_y.
\end{cases}
\end{equation}
The operator $P$ has the symbol 
\begin{equation}\label{eq6a.1}
	p(x,y;\xi,\eta):= \xi +i\eta + \partial_y\varphi =: p_\xi(y,\eta),
\end{equation}
defined on $S^1_x\times Y \times \R_\xi\times \R_\eta$. Correspondingly 
$P$ has the orthogonal direct sum decomposition 
\begin{equation}\label{eq:int1}
	P = \bigoplus_{\xi\in h\Z}P_\xi, 
	\quad 
	P_\xi = \xi +h\partial_y + \partial_y\varphi \text{ on } Y, 
	\quad h\in ]0,h_0].
\end{equation}
The characteristic set of $P$ is given by 
\begin{equation}\label{eq:int1.1}
	p^{-1}(0)= 
	\{
		(x,y;\xi,\eta)\in S^1\times Y\times \R_\xi \times \R_\eta; ~\eta =0,
		~\partial_y\varphi =-\xi
	\}.
\end{equation}
When $Y=\R$ we have $\xi\leq 0$ on $p^{-1}(0)$ and when $Y=S^1$ we have 
$\xi\in[-1,1]$. We exclude the limiting values $\xi=0$ and $\xi\in \{-1,1\}$ 
respectively and put 
\begin{equation}\label{eq:int1.4}
	\Sigma =
	\{
		(x,y;\xi,\eta)\in	p^{-1}(0); ~ \xi <0 \text{ when } Y=\R,~
		-1 < \xi <1 \text{ when } Y=S^1
	\}.
\end{equation}
We recall that that the Hamilton vector field of a $C^1$ function 
$f(x,y,\xi,\eta)$ is given by
\begin{equation*}
H_f= \partial _\xi f\, \partial _x+\partial _\eta f\, \partial
_y-\partial _xf\, \partial _\xi -\partial _yf\, \partial _\eta .
\end{equation*}
Then we get 
\begin{equation}\label{eq:int1.3}
  \frac{1}{2i} \{ p,\overline{p}\} =
  \frac{1}{2i}H_p(\overline{p})=
	\frac{1}{2i}\left(
		\partial _\xi p\, \partial _x\overline{p}
		+\partial _\eta p\, \partial_y\overline{p} 
		-\partial _x p\, \partial _\xi\overline{p} 
		-\partial _y p\, \partial _\eta \overline{p}
	\right)
	= \partial^2_y\phi
\end{equation}
where $\{\cdot,\cdot\}$ denotes the Poisson bracket.
We split $\Sigma=\Sigma_+\cup\Sigma_-$, where 
\begin{equation}\label{eq:int1.2.5}
\begin{split}
	\Sigma_\pm =
	\{
		(x,y;\xi,\eta)\in	p^{-1}(0); ~ y = y_\pm(\xi),~ 
		&\xi \in ]-\infty, 0[ \text{ when } Y=\R,\\
		&\xi\in ]-1,1[ \text{ when } Y=S^1
	\}.
\end{split}
\end{equation}
where $y_\pm(\xi)$ are the solutions to $\partial_y \phi(y_\pm(\xi))=-\xi$ 
with $\pm\partial^2_y\phi(y_\pm(\xi)) >0$. Observe that 
$\frac{1}{2i}\{p_\xi,\overline{p_\xi}\}=\partial^2_y\phi$ and that 
$\{p,\overline{p}\}(x,y;\xi,\eta)=\{p_\xi,\overline{p_\xi}\}(y,\eta)$, so 
\begin{equation}\label{eq:int1.2}
	\Sigma_\pm =
	\{
		(x,y;\xi,\eta)\in	p^{-1}(0); ~ 
		\pm\frac{1}{i}\{p,\overline{p}\}(x,y;\xi,\eta) >0
	\}.
\end{equation}
The submanifolds $\Sigma_\pm$ are symplectic. Indeed let 
$\sigma =d\xi\wedge dx + d\eta\wedge dy$ be the symplectic form on 
$T^*X\times Y$. Then, 
\begin{equation}\label{eq:int1.3}
	\sigma|_{\Sigma_\pm} \simeq d\xi\wedge dx,
\end{equation}
since $\eta =0$ on $\Sigma_\pm$.
\section{The main result}
We are interested in the small singular values of $P$ which are 
$\ll h^{2/3}$. We recall that the singular spectrum of $P$ is defined 
as the square root of the singular spectrum of $P^*P$ which is a positive 
essentially self-adjoint operator and we equip it with its natural domain. Indeed, when 
$Y=S^1$, then by the Rellich theorem we know that the inclusion 
$H^2_h(S^1\times S^1)\xhookrightarrow{} L^2(S^1\times S^1)$ is compact which implies 
that $P^*P$ has a compact resolvent and so the spectrum of $P^*P$ is purely 
discrete and $\subset [0,\infty[$. When $Y=\R$, we will see by a less direct 
argument that the spectrum $P^*P$, being contained in $[0,\infty[$, is discrete 
away from $0$. The \emph{singular values} of $P$ are then defined as the square roots 
of the eigenvalues of $P^*P$. 
\par
We define the partial Fourier transform of a function 
$u(x,y)$ on $X\times Y$ by
\begin{equation*}
\mathcal{F}u(\xi ,y)=\widehat{u}(\xi ,y)=\int_Xe^{-ix\xi /h}u(x,y)dx, \ \xi
\in \widehat{X}= h \mathbb{Z}.
\end{equation*}
Then $\mathcal{F}:\, L^2(X\times Y)\to L^2(\widehat{X}\times Y)$ is
unitary when $X$ and $Y$ are equipped with the Lebesgue measure, 
and $\widehat{X}=h\mathbb{Z}$ is equipped with the $1/(2\pi )$ times 
the counting measure on $h\mathbb{Z}$.
%
%
\par
After applying $\mathcal{F}$, the equation $Pu=v$ becomes
\begin{equation}\label{eq9} 
	P_\xi \widehat{u}(\xi,y) = \widehat{v}(\xi,y),\hbox{ where }
        P_\xi= h\partial_y + \xi + \partial_y\phi,\ \xi \in
        \widehat{X},\ y\in Y, \quad h\in ]0,h_0],
\end{equation}
and we deduce \eqref{eq:int1}. We equip the operator $P_\xi$ with its 
natural domain given by the semiclassical Sobolev space
\begin{equation}\label{eq9.1}
\mathcal{D}(P_\xi)
=\{ u\in
L^2(\R);\, (1+y^4)^{1/2}u,\, hD_yu\in L^2({\R}) \}
\end{equation}
when $Y=\R$, and
\begin{equation*}
	\mathcal{D}(P_\xi)=H_h^1(S^1)=\{u\in L^2(S^1);\, hD_yu\in L^2(S^1) \}, 
\end{equation*}
when $Y=S^1$. See Section \ref{sec:semclass} below for a brief review of 
semiclassical pseudo-differential operators. In fact $P_\xi$ is an
elliptic pseudo-differential operator in its natural class and it will 
follow that $P_\xi :\, \mathcal{D}(P_\xi )\to L^2(Y)$ is a Fredholm
operator of index $0$, (can also be verified with direct ODE arguments). 
This extends to the adjoint
\begin{equation}\label{eq9a}
	P_\xi^* = -h\partial_y + \xi + \partial_y\phi ,
\end{equation}
which has the same domain as $P_\xi $.
\par 
The principal symbol $p_\xi$ of $P_\xi$ is given by 
\begin{equation}\label{eq9b}
	p_\xi(y,\eta ) = i\eta+ \xi + \partial_y\phi(y), \quad (y,\eta)\in T^*Y, 
\end{equation}
see also \eqref{eq6a.1}, and the one of $P_\xi ^*$ is equal to 
$\overline{p}_\xi $. Both operators are of principal type in the sense that 
$dp_{\xi} \neq 0$ and $d\overline{p}_{\xi} \neq 0$. The common characteristic 
set $p_\xi ^{-1}(0)$ of $P_\xi$ and $P_\xi^*$ in \eqref{eq9}, \eqref{eq9a} 
is given by 
\begin{equation}\label{eq10}
	\eta =0, ~~ \xi + \partial_y\phi=0.
\end{equation}
\subsection{Weyl asymptotics}
In Section \ref{sec:RE} we establish that the spectrum of 
$P^*_\xi P_\xi$, $\xi\in h\Z$, is purely discrete (in both case $Y=S^1$ and 
$Y=\R$) and that it is equal to the spectrum of $P_\xi P_\xi^*$. We will denote 
the singular values of $P_\xi$ by $t_k(\xi)$, $k\in\N$, counting with their multiplicity  
and ordered in an increasing way, i.e. 
\begin{equation}\label{eq10.1}
	0\leq t_0(\xi) \leq t_1(\xi) \leq \dots \to +\infty. 
\end{equation}
In Section \ref{sec:RE}, we will establish resolvent estimates for $\xi$ such that 
$P_\xi$ is elliptic or close to elliptic and obtain the following 
\begin{thm}\label{thm:main0}
We define $\phi $ as in \eqref{eq6a} and $P_\xi =h\partial _y+\xi
+\partial _y\phi $ as in \eqref{eq9}.
\par
1. Let $Y=\R$. For every $C_0>0$, there exists a constant 
$C>0$ such that if $-C_0 h^{2/3} \leq \xi$ and $h_0>0$ is small 
enough, then the smallest singular value $t_0(\xi)$ of $P_\xi$ 
satisfies 
\begin{equation*}
	t_0(P_\xi) \geq \frac{1}{C}(|\xi|+h^{2/3}).
\end{equation*}
\par 
2. Let $Y=\R/2\pi \Z$. For every $C_0>0$, there exists a constant 
$C>0$ such that if $1-C_0h^{2/3}\leq |\xi|$ and $h_0>0$ is small 
enough, then the smallest singular value $t_0(\xi)$ of $P_\xi$ 
satisfies 
\begin{equation*}
	t_0(P_\xi) \geq \frac{1}{C}(|\xi|+h^{2/3}).
\end{equation*}
\end{thm} 
From a microlocal persepective, the above result is based on the 
fact that when $\xi>0$ when $Y=\R$ or $|\xi|>1$ when $Y=S^1$, the 
the prinicpal symbol of $P_\xi$ is in modulus $>0$ and the result 
follows from a suitable version of G\aa{}rding's inequality. We will 
extend the resulting estimates to the region $|\xi|\leq O(h^{2/3})$ when 
$Y=\R$ and to $||\xi|-1|\leq O(h^{2/3})$ by conjugation with a suitable 
weight function. 
\begin{rem}
	When $\xi=0$, we have that $0$ is a point in the boundary of 
	$\overline{p_\xi(T^*Y)}$ and $\{p,\overline{p}\}(0,0)=0$. 
	The symbol $p_\xi$ satisfies the following subellipticity 
	condition  
	\begin{equation*}
		H_{\Ima p_\xi}^2\Rea p_\xi(0,0) 
		= \partial^3_y\phi(0) \neq 0
	\end{equation*}
	where $H_{\Ima p_\xi}$ denotes the Hamilton vector 
	field induced by $\Ima p_\xi$. This, by \cite[Theorem 1.4]{DSZ04} and 
	\cite[Theorem 1.1]{Sj10a} (see also the references therein), implies that 
	\begin{equation}\label{We:e1.1}
		\|P_0^{-1}\|_{L^2(\R)\to L^2(\R)} = O(h^{-2/3}).
	\end{equation}
	Using a Neumann series argument \eqref{We:e1.1} immediately 
	implies that
	\begin{equation}\label{We:e1.2}
		\|P_\xi^{-1}\|_{L^2(\R)\to L^2(\R)} = O(h^{-2/3}), 
		\quad \text{for } |\xi| \ll h^{2/3}.
	\end{equation}
	For our model operators we will use more direct arguments.
\end{rem}
Next, we study the ``non-elliptic'' region, that is to say that 
we study the region where $-\xi \gg h^{2/3}$ when $Y=\R$ and 
$h^{2/3} \ll 1- |\xi|$ when $Y=S^1$.  
In Section \ref{WL} we will show that the smallest singular value 
$t_0(\xi)$ of $P_\xi$ is of the form $|m_+(\xi)|$ where $m_+(\xi)$ 
has the following asymptotics.
\begin{thm}\label{thm:main1}
We define $\phi $ as in \eqref{eq6a} and $P_\xi =h\partial _y+\xi
+\partial _y\phi $ as in \eqref{eq9} with $h_0>0$ small enough. For $\xi < 0$, 
when $Y=\R$, and $\xi \in ]-1,1[$ when $Y=S^1$, let $y_+,\, y_-\in Y$ 
be the two solutions of the equation $\partial_y\phi (y)=-\xi $, labelled 
so that $\pm \partial _y^2\phi (y_\pm)>0$. Let $d$ denote the Lithner-Agmon 
distance on $Y$ for the metric $(\xi+\partial _y\phi (y))^2dy^2$ and define 
the action 
\begin{equation*}
	S_0:\mathcal{D}(S_0):= \begin{cases}
		]-\infty,0[, \quad Y=\R,\\
		]-1,0[,\quad Y=S^1
	\end{cases}
	\rightarrow ~~
		]0,+\infty[, ~~ \xi \mapsto S_0(\xi)=d(y_+(\xi),y_-(\xi)).
\end{equation*} 
Then, uniformly with respect to $\xi$ varying in a compact 
$h$-independent subset of $]-\infty,0[$ when $Y=\R$, and of $]-1,1[$ when 
$Y=S^1$, the smallest singular value of $P_\xi$ is equal to $|m_+|$ where 
for any $\varepsilon>0$ 
\begin{equation*}
\begin{split}
	m_+=&h^{\frac{1}{2}}\left( \left( \frac{|\{p_\xi ,\overline{p}_\xi
		\}(y_+,0)|}{4\pi } \right)^{1/4}
	\left( \frac{|\{p_\xi ,\overline{p}_\xi  \}(y_-,0)|}{4\pi } \right)^{1/4}
	+\mathcal{ O}(h)\right)\\ &
	\times  \begin{cases}e^{-S_0/h}+\mathcal{O}_\varepsilon(e^{\varepsilon/h-3S_0/h})
		,\hbox{ when }Y={\R}\\
	\left(e^{-\frac{1}{h}d_J(y_+,y_--2\pi ) }-e^{-\frac{1}{h}d_J(y_+,y_-)} \right)
	+\mathcal{O}_\varepsilon(e^{\varepsilon/h-3S_0/h}),\hbox{ when }Y=S^1.
	\end{cases}
\end{split}
\end{equation*}
In the latter case we identify $y_+$, $y_-$ with points in 
$]-\pi ,0[$ and $]0,\pi [$ respectively and let $d_J$ denote the 
Lithner-Agmon distance on $J=]y_--2\pi,y_-[$. 
\par 
Furthermore, for any $C_0>0$ if $-C_0h^2 \geq \xi \geq -1/C_0$ or $-C_0 \geq \xi$ 
when $Y=\R$, or if $C_0h^{2/3} \leq 1\pm \xi \leq 1/C_0$ when $Y=S^1$, then 
for $h_0>0$ small enough 
\begin{equation*}
\begin{split}
m_+&=
h^{\frac{1}{2}}\left( \left( \frac{|\{p_\xi ,\overline{p}_\xi
    \}(y_+,0)|}{4\pi } \right)^{1/4}
\left( \frac{|\{p_\xi ,\overline{p}_\xi  \}(y_-,0)|}{4\pi } \right)^{1/4}
+\mathcal{ O}\!\left(h |\xi|^{-5/4}\right)\right)
e^{-S_0/h}
\\
&=\frac{\sqrt{h}|\xi|^{1/4}}{\sqrt{\pi}}
\left( 1 +\mathcal{ O}(h|\xi|^{-3/2})\right)e^{-4|\xi|^{3/2}/3h}.
\end{split}
\end{equation*} 
\end{thm}
We note here that the Lithner-Agmon distance $d$ on $Y$ for the metric $(\xi
+\partial _y\phi (y))^2dy^2$ satisfies 
\begin{equation*}
	d(y_\pm ,y)=\pm (f(y)-f(y_\pm )),\ \ y\in \mathrm{neigh\,}(y_\pm),
\end{equation*}
with $f(y,\xi )=y\xi + \phi (y)$.
\\
\par 
On the other hand, the second smallest singular values of $P_\xi$ 
satisfies $t_1(P_\xi) \geq h^{2/3}/C$ for all $\xi\in \R$, see also 
Theorem \ref{thm:main0}. More precisely,  
\begin{thm}\label{thm:main0.1}
	Under the assumption of Theorem \ref{thm:main0}, we have that 
	when $\xi$ varies in a compact $h$-independent subset 
	of $]-\infty,0[$ when $Y=\R$, and of $]-1,1[$ when $Y=S^1$, there exists 
	a constant $C>0$ such that for $h_0>0$ small enough 
	$t_1(P_\xi)\geq h^{1/2}/C$ uniformly in $\xi$. 
	\par
	For every $C_0>0$ if $C_0h^{2/3}\leq -\xi \leq 1/C_0$ or $-\xi \geq C_0$ 
	when $Y=\R$, and if	$C_0h^{2/3}\leq 1 + \xi \leq 1/C_0$ or 
	$C_0h^{2/3}\leq 1 - \xi \leq 1/C_0$ when $Y=S^1$, there exists a constant 
	$C>0$ such that for $h_0>0$ small enough 
	$t_1(P_\xi)\geq h^{1/2}|\xi|^{1/4}/C$. 
\end{thm}
From the orthogonal decomposition \eqref{eq:int1} we see that the singular 
values of $P$ are of the form $t_k(\xi)$, $k\in\N$, $\xi\in h\Z$, where 
$t_k(\xi)$ denote the singular values of $P_\xi$. By Theorems \ref{thm:main0}, 
\ref{thm:main1} and \ref{thm:main0.1}, the singular values of $P$ in 
$[\e^{-1/C h}, h^{2/3}/C]$ when $Y=\R$ and in $[0,h^{2/3}/C]$ when $Y=S^1$, 
for $C>0$ sufficiently large and $h_0>0$ small enough, 
are of the form $t_0(\xi)=|\mu_+(\xi)|$ where
\begin{equation*}
	\xi \in h\Z \text{ satisfies } 
	\begin{cases}
		C_0h^{2/3} \leq -\xi \leq C_0, \quad Y=\R, \\
		C_0h^{2/3} \leq 1-|\xi|, \quad Y=S^1,
	\end{cases}
\end{equation*}
for $C_0>0$ sufficiently large. 
\par 
Recall \eqref{eq:int1.2.5}, \eqref{eq:int1.2}. The action $S_0$ defined Theorem 
\ref{thm:main1} can be seen as the application 
\begin{equation*}
	S_0: \Sigma_+ \rightarrow [0,+\infty[
\end{equation*}
Our problem is then to study the distribution of the values $t_0(\xi)$. Here 
the counting measure $\sum_{\xi\in h\Z}\delta_\xi$ appears naturally as well as its 
approximation in $\frac{1}{h}d\xi$. Because of the translation invariance in 
$x$, we can use the measure $\frac{1}{2\pi h}dx\otimes d\xi$ on $S^1\times \R_\xi$ 
and $dx\otimes d\xi$ can be identified with the symplectic volume element 
$\sigma|_{\Sigma_+}$ on $\Sigma_+$, see \eqref{eq:int1.2}, \eqref{eq:int1.3}. 
Following this argument in Section \ref{sec:CountingSV} we get
\begin{thm}\label{thm:main2}
Let $P$ be \eqref{eq1}, let $\Sigma_+$ be as in \eqref{eq:int1.2} and recall 
\eqref{eq:int1.3}. Let $S_0$ be as in Theorem \ref{thm:main1}. Let $C_0>0$ 
be large enough and let 
\begin{equation*}
	C_0h \leq \frac{\delta^{3/2}}{\log \delta^{-1}}, \quad \delta>0. 
\end{equation*}
Then, for $0<a<b$ with $b\asymp 1$ and $a \asymp \delta^{3/2}$,  
\begin{equation*}
	\left|\#\left(
		\mathrm{Spec}(P^*P)\cap [\e^{-b/h},\e^{-a/h}]
	 \right) 
	 -
	\frac{1}{2\pi h} \int_{S_0^{-1}([a,b])} d\sigma|_{\Sigma_+}
	\right|
	=\mathcal{O}(1)\frac{\log \delta^{-1}}{\sqrt{\delta}}.
\end{equation*}
\end{thm}
\subsection{Possible generalizations}\label{poss_gen}
We will work on the natural complexification of the cotangent space
$X\times Y\times \R^2$ and we let $p$ also denote the holomorphic
extension of $p$ to this space. The holomorphic extension of
$\overline{p}$ is then given by $p^*(\rho
):=\overline{p(\overline{\rho })}$. Let $\mathcal{ P}$, $\mathcal{ P}^*$
denote the complex hypersurfaces given by $p=0$ and $p^*=0$
respectively.
\par 
Near a point in $\Sigma _\pm$, the complexification $\Sigma
_\pm^{\C}$ of $\Sigma _\pm$ is the complex manifold given by 
$p(\rho )=p^*(\rho)=0$. $\Sigma _\pm^\C$ are complex symplectic 
manifolds of
codimension 2 and of dimension 2, which locally coincide with $\mathcal{
  P}\cap \mathcal{ P}^*$. Locally we can view $\Sigma _\pm^\C$ as the
complex hypersurface in $\mathcal{ P}$ given by $p^*=0$ and also as the
complex hypersurfice in $\mathcal{ P}^*$, given by $p=0$.

The complex vector fields $H_p$, $H_{p^*}$ are tangent to $\mathcal{ P}$
and $\mathcal{ P}^*$ respectively, and their restrictions to $\mathcal{ P}$
and $\mathcal{ P}^*$ respectively are transversal to $\Sigma _\pm ^\C$, 
since
$$
H_pp^*=-H_{p^*}p\ne 0.
$$
If $\rho \in \Sigma _+^\C$ is close to $\Sigma _+$, the
complex $H_p$ integral curve in $\mathcal{ P}$ through $\rho $ intersects
$\Sigma _-^\C$ at a point $\kappa (\rho )\in \Sigma _-^\C$
and $\kappa :\Sigma _+^\C\to \Sigma _-^\C$ is a  canonical
transformation. The situation is stable under small perturbations. We
believe that the results in the paper can be formulated with the help
of $\kappa $ and action integrals along $H_p$ and $H_{p^*}$ integral
curves that connect $\Sigma _+^\C$ and $\Sigma _-^\C$. Moreover, 
this may be the right point of view for generalizations. 
\subsection{Notation}
We will frequently use the following notation: when we write $a \ll b$, we 
mean that $Ca \leq  b$ for some sufficiently large constant $C>0$. The 
notation $f = \mO(h^\alpha)$, $\alpha\in\R$, means that there exists a 
constant $C>0$ (independent of $h$>0) such that $|f| \leq C h^\alpha$. When 
we want to emphasize that the constant $C>0$ depends on some parameter $k$, 
then we write $C_k$, or with the above notation $\mO_k(N)$. When we want 
to emphasize that $f$ is positive, we write $f=1/\mO(1)$ then we mean that 
there exists a constant $C>0$ such that $0\leq f \leq 1/C$.
\\
\\
\paragraph{\textbf{Acknowledgments}} 
We thank M. Hitrik and M. Zworski for interesting discussions, and 
refer to the reader to their related work in progress \cite{HiZw23+}.
M. Vogel was partially funded by the Agence Nationale de la 
Recherche, through the project ADYCT (ANR-20-CE40-0017).
\section{A priori resolvent estimates and scaling}
\subsection{Semiclassical quantization}\label{sec:semclass}
We begin by reviewing some basic notions and properties 
of the calculus of semiclassical pseudodifferential operators on $\R^d$, 
with $d\in\N\backslash\{0\}$. For the results and notions reviewed here, 
as well as for a further in depth reading, we refer the reader to 
\cite{GriSjoMLA,DiSj99,Ma02,Zw12,Ho84}. 
\par
Let $h\in]0,1]$ denote the \emph{semiclassical parameter}. We call a 
function $m\in C^\infty(\R^{2d};]0,\infty[)$ an \emph{order function}, 
if there exist constants $C,N>0$ such that 
\begin{equation*}
	m(\rho) \leq C(1+|\rho-\mu|^2)^{N/2}m(\mu), \quad \text{for all } \rho,\mu\in \R^{2d}.
\end{equation*}
Given an order function $m$, we define the class $S(m)$ of possibly $h$-dependent 
\emph{symbols} as the class of complex-valued functions $p(\rho;h)$ on 
$\R^{2d}_\rho\times ]0,1]_h$, such that $p(\cdot;h) \in C^\infty(\R^{2d})$, 
for all $h\in]0,1]$, and such that $p$ satisfies the following derivative bounds  
\begin{equation*}
	\forall \alpha \in \N^{2d}: ~\sup\limits_{h\in]0,1]} \,\sup\limits_{\rho\in\R^{2d}}
	m(\rho)^{-1}|\partial^\alpha p(\rho;h)| < +\infty.
\end{equation*}
A symbol $p\in S(m)$ gives rise to a linear operator 
$\Op_h(p):\mathcal{S}(\R^d)\to\mathcal{S}'(\R^d)$ called the $h$-\emph{Weyl quantization} 
of $p$, defined by the oscillatory integral
\begin{equation}\label{eq:sq2} 
\begin{split}
	\Op_h^w(p)u(x)
	&:=p^w(x,hD_x;h) u(x) \\
	&= \frac{1}{(2\pi h)^d } \iint_{\R^{2d}} \e^{\frac{i}{h}\xi(x-y)} 
	p\!\left( \frac{x+y}{2}, \xi;h\right) u(y) dy d\xi.
\end{split}
\end{equation}
Here, $\mathcal{S}(\R^{d})$ denotes the space of Schwartz test functions on $\R^d$, and 
$\mathcal{S}'(\R^{d})$ its topological dual, the space of tempered distributions $\R^d$. 
It can be shown that $p^w$ is continuous $\mathcal{S}(\R^{d}) \to \mathcal{S}(\R^{d})$, 
and therefore continuous $\mathcal{S}'(\R^{d}) \to \mathcal{S}'(\R^{d})$ by duality. When the 
order function $m$ is equal to $1$, then $p^w$ is a bounded operator $L^2(\R^{d}) \to L^2(\R^{d})$. 
Furthermore, when $m(\rho)\to 0$ as $|\rho|\to \infty$, then $p^w:L^2(\R^{d}) \to L^2(\R^{d})$ 
is compact. 
\subsection{Resolvent estimates}\label{sec:RE}
We are interested in the small singular values of $P_\xi$ and 
$P_\xi^*$, see \eqref{eq9}. More precisely we shall study the bottom of the spectrum 
of the semiclassical differential operators $P_\xi^*P_\xi$ and 
$P_\xi P_\xi^*$ which are elliptic away from the set \eqref{eq10}. 
Both operators are essentially selfadjoint with domain 
\begin{equation}\label{neq:1}
\begin{split}
&\mathcal{D}(P_\xi P_\xi^*) 
	=\mathcal{D}(P_\xi^*P_\xi ) \\
	&=\begin{cases}
		\{u\in L^2(Y); hD_yu, ~(hD_y)^2u,~(1+y^2)^2u\in L^2(Y)\}, 
			~\text{when }Y=\R; \\
		H^2(S^1)=\{u\in L^2(S^1); hD_yu, ~(hD_y)^2u \in L^2(S^1)\},
			~ \text{when }Y=S^1.
	\end{cases}
\end{split}
\end{equation}
Notice that $P_\xi^*P_\xi$ and $P_\xi P_\xi^*$ have purely discrete 
spectrum $\subset [0,+\infty[$ and we know (see for instance 
\cite{GoKr69}) that the spectra $\sigma(P_\xi^*P_\xi)$ and 
$\sigma(P_\xi P_\xi^*)$ coincide away from $0$. In particular 
the two spectra coincide completely when both $P_\xi$ and $P_\xi^*$ 
are injective. Denoting the eigenvalues of $P_\xi^*P_\xi$ and 
$P_\xi P_\xi^*$ by $t_j^2(P_\xi)$ and $t_j^2(P_\xi^*)$ respectively, 
arranged in increasing order with $t_j(P_\xi)\geq 0$, 
$t_j(P_\xi^*)\geq 0$, the \emph{singular values} of $P_\xi^*$ 
and $P_\xi$ are by definition equal to $t_j(P_\xi)$ and 
$t_j(P_\xi^*)$ respectively. As remarked above, the non-vanishing 
singular values concide. 
%
%
\\
\par 
\textbf{Injectivity of $P_\xi$ and $P_\xi^*$.} 
We begin with the case $Y=\R$. If $0\neq u \in \mathcal{N}(P_\xi)$, we get 
$P_\xi u(y) =0$, i.e. 
\begin{equation*}
		u(y) = C \e^{-\frac{1}{h}(\xi y + \phi(y))}.
\end{equation*}
By \eqref{eq6a} this function cannot belong to $L^2(\R)$ when $C\neq 0$. 
Hence, $\mathcal{N}(P_\xi)=0$ in this case. 
\par
Next we turn to the case $Y=S^1$. In this case an element 
$u\in\mathcal{N}(P_\xi)$ is again of the form 
$u(y) = C \exp\{-\frac{1}{h}(\xi y + \phi(y))\}$ and has to be 
$2\pi$-periodic. Now $\varphi(y)=\sin y$ is $2\pi$-periodic by 
\eqref{eq6a} and since $\xi$ is real the $2\pi$-periodicity implies 
that $C=0$ when $\xi\neq0$. Hence, in this case $\mathcal{N}(P_\xi)=0$ 
when $\xi\neq 0$, and $\mathcal{N}(P_0)=\C\cdot\exp\{-\frac{1}{h}\phi(y)\}$. 
\\
\par 
Similarly $P_\xi^*$ is injective for all 
$\xi\in\R$, when $Y=\R$ and for all $\xi\in\R\backslash\{0\}$ when $Y=S^1$, 
while $\mathcal{N}(P_0^*)= \C\cdot\exp\{\frac{1}{h}\phi(y)\}$ in the latter 
case. 
\par 
Since the kernel $\mathcal{N}(P_\xi^*P_\xi)=\mathcal{N}(P_\xi)$ and 
$\mathcal{N}(P_\xi P_\xi^*)=\mathcal{N}(P_\xi^*)$, we conclude that the 
spectra $\mathrm{Spec}(P_\xi^*P_\xi)$ and $\mathrm{Spec}(P_\xi P_\xi^*)$ coincide for 
all $\xi \in \R$. Hence the singular values of $P_\xi$ and $P_\xi^*$ coincide 
for all $\xi \in \R$.
\\
\par 
Next we prove the following resolvent etstimates. 
\begin{prop}\label{prop1}
Let $P_\xi$ be defined as in \eqref{eq9}. 
\\
1. When $Y=\R$ and $\phi = y^3/3$. Then, for any fixed $C_0>0$ 
	\begin{equation}\label{prop1.1}
		\|P_\xi^{-1}\|_{L^2(\R)\to L^2(\R)} = O(h^{-2/3}), 
		\quad |\xi| \leq C_0h^{2/3},
	\end{equation}
	and 
	\begin{equation}\label{prop1.2}
		\|P_\xi^{-1}\|_{L^2(\R)\to L^2(\R)} = O(|\xi|^{-1}), 
		\quad C_0h^{2/3} \leq \xi.
	\end{equation}
2. When $Y=\R/2\pi \Z$ and $\phi = \sin y$. Then, for any fixed $C_0>0$ 
\begin{equation}\label{prop1.3}
	\|P_\xi^{-1}\|_{L^2(\R)\to L^2(\R)} = O(h^{-2/3}), 
	\quad ||\xi| -1|\leq C_0 h^{2/3},
\end{equation}
and 
\begin{equation}\label{prop1.4}
	\|P_\xi^{-1}\|_{L^2(\R)\to L^2(\R)} = O(|\xi|^{-1}), 
	\quad 1+C_0 h^{2/3} \leq |\xi|.
\end{equation}
\end{prop}
Before we turn to the proof of this proposition, we state an  
immediate consequence. Indeed, since the square of the 
singular values of $P_\xi$, see \eqref{eq10.1}, are the eigenvalues 
of the self-adjoint operator $P_\xi^* P_\xi$, the spectral theorem 
and Proposition \ref{prop1} imply Theorem \ref{thm:main0}. 
%
%
\begin{proof}[Proof of Position \ref{prop1}]
We will only consider the case $Y=\R$ and $\phi(y)=y^3/3$, 
as the other case is similar. 
\\
\par
1. Let $\psi$ be a real-valued $C^1$ function on $\R$ such that $\psi$ and 
$\partial_y\psi$ are bounded and consider 
\begin{equation*}
	\e^{-\psi/h}P_\xi\e^{\psi/h} 
	= h\partial_y +\xi +\partial_y\phi +\partial_y\psi.
\end{equation*}
We find by integration by parts that for all $u\in\mathcal{D}(P_\xi)$, 
see \eqref{eq9.1}, 
\begin{equation}\label{eq:REpr0}
	\Rea \left( \e^{-\psi/h}P_\xi\e^{\psi/h} u | u \right)
	= 
	\left( (\xi +\partial_y\phi +\partial_y\psi) u | u \right)
	\geq 
	(\xi + \inf_{y\in\R}(\partial_y\phi +\partial_y\psi) )
	\| u\|^2.
\end{equation}
Assume that 
\begin{equation}\label{eq:REpr1}
	\xi + \inf_{y\in\R}(\partial_y\phi +\partial_y\psi)  >0.
\end{equation}
Then, upon estimating the left hand side of \eqref{eq:REpr0} 
from above by $\|\e^{-\psi/h}P_\xi\e^{\psi/h} u\| \| u \|$, 
we get 
\begin{equation*}
	(\xi + \inf_{y\in\R}(\partial_y\phi +\partial_y\psi) )\| u\|
	\leq 
	\|\e^{-\psi/h}P_\xi\e^{\psi/h} u\|,
\end{equation*}
and therefore 
\begin{equation}\label{eq:REpr2}
	\|\e^{-\psi/h}P_\xi^{-1} \e^{\psi/h}\| 
	\leq 
	(\xi + \inf_{y\in\R}(\partial_y\phi +\partial_y\psi) )^{-1}.
\end{equation}
Writing $P_\xi^{-1} = \e^{\psi/h}(\e^{-\psi/h}P_\xi^{-1} 
\e^{\psi/h})\e^{-\psi/h}$, \eqref{eq:REpr2} implies that 
\begin{equation*}
	\|P_\xi^{-1}\| 	
	\leq 
	\e^{\sup_y \psi/h}
	(\xi + \inf_{y\in\R}(\partial_y\phi +\partial_y\psi) )^{-1}
	\e^{-\inf_y\psi/h}.
\end{equation*}
Thus, 
\begin{equation}\label{eq:REpr3}
	\|P_\xi^{-1} \| 
	\leq 
	\frac{\e^{\frac{1}{h}(\sup_y \psi-\inf_y \psi)}}
	{\xi + \inf_{y\in\R}(\partial_y\phi +\partial_y\psi) }.
\end{equation}
2. Next we choose $\psi$. Let $\alpha>0$ and take
\begin{equation}\label{eq:REpr4}
	\psi(y)
	=
	\begin{cases}
	0, \quad \text{for } y \leq - \sqrt{\alpha}, \\
	\alpha y - \frac{1}{3}y^3 + \frac{2}{3}\alpha^{3/2}, 
		\quad \text{for }- \sqrt{\alpha} \leq y \leq \sqrt{\alpha}, \\ 
	\frac{4}{3}\alpha^{3/2}, \quad\text{for } y \geq \sqrt{\alpha}.
	\end{cases}
\end{equation}
Then $\inf_y(\partial_y\phi + \partial_y\psi)=\alpha$, 
$\sup_y \psi - \inf_y \psi = \frac{4}{3}\alpha^{3/2}$ and we 
get by \eqref{eq:REpr3} 
\begin{equation}\label{eq:REpr5}
	\|P_\xi^{-1} \| 
	\leq 
	\frac{\e^{\frac{4}{3h}\alpha^{3/2}}}
	{\xi + \alpha }, \quad \text{when } \alpha > - \xi. 
\end{equation}
3. Next, we look for a good choice of $\alpha$ when $\xi \geq 0$. 
Differentiating the right hand side of \eqref{eq:REpr5} with respect 
to $\alpha$ we see that it has a critical point when 
$(\alpha +\xi)\alpha^{1/2} = h/2$. This equation has a unique solution 
corresponding to the point of minimum of 
$]0,+\infty[\ni\alpha \mapsto (\xi+\alpha)^{-1} \e^{\frac{4}{3h}\alpha^{3/2}}$. 
Define 
\begin{equation}\label{eq:REpr6}
	f(\alpha,\xi):= \min (2\xi\alpha^{1/2},2\alpha^{3/2}).
\end{equation}
Note that $f(\cdot,\xi)$ is an increasing function such that 
$f(\alpha,\xi)\leq (\alpha+\xi)\alpha^{1/2}$ with equality when 
$\alpha=\xi$. We solve the simplified equation $f(\alpha,\xi) = h/2$ 
and get 
\begin{equation}\label{eq:REpr7}
	\alpha = 
	\begin{cases}
		\left(\frac{h}{4\xi}\right)^2\leq \xi \quad \text{when } \xi 
			\geq \left(\frac{h}{4}\right)^{2/3}\\
		\left(\frac{h}{4}\right)^{2/3}\geq \xi, 
		\quad \text{when }\xi\leq\left(\frac{h}{4}\right)^{2/3}.
	\end{cases}
\end{equation}
Using that $\xi+\alpha\geq \xi$, $\xi+\alpha \geq \alpha$ respectively 
in the two cases, we get 
\begin{equation}\label{eq:REpr8}
	\| P_\xi^{-1}\| \leq 
	\begin{cases}
	\frac{1}{\xi}\e^{\frac{h^2}{3\cdot 4^2 \xi^3}} 
		= \frac{O(1)}{\xi}, \quad 
		\text{when }\xi \geq \left(\frac{h}{4}\right)^{2/3},\\ 
	\frac{1}{(h/4)^{2/3}}\e^{\frac{4}{3h}(h/4)} 
		=
		\frac{O(1)}{\xi+h^{2/3}}, 
		\quad \text{when }\xi\leq \left(\frac{h}{4}\right)^{2/3}.
	\end{cases}
\end{equation}
%
4. Next, we look for an optimal choice of $\alpha$ when 
$-Ch^{2/3}\leq \xi <0$. To simplify the equation 
$(\alpha +\xi)\alpha^{1/2}=h/2$ we approximate it with $g(\alpha,\xi)=h/2$ 
where 
\begin{equation}\label{eq:REpr10}
	g(\alpha,\xi)
	=
	\begin{cases}
		(\alpha+\xi)(-2\xi)^{1/2}, \quad \text{when } -\xi\leq\alpha\leq -2\xi, \\
		\alpha^{3/2}/2,  
		\quad \text{when }-2\xi\leq \alpha.
	\end{cases}
\end{equation}
Note that $(\alpha+\xi)\alpha^{1/2}$ and $g(\alpha,\xi)$ are increasing 
continuous functions of $\alpha\in[-\xi,+\infty[$ which coincide at 
$\alpha=-2\xi$. The solution to the equation $g(\alpha,\xi)=h/2$ is given 
by 
\begin{equation*}
	\alpha
	=
	\begin{cases}
	-\xi +2^{-3/2}h(-\xi)^{-1/2},\text{ when } -\xi\geq h^{2/3}/2, \\
	h^{2/3},\text{ when } -\xi\leq h^{2/3}/2.
	\end{cases}
\end{equation*}
With this choice of $\alpha$ we get by \eqref{eq:REpr5} 
\begin{equation}\label{eq:REpr11}
	\| P_\xi^{-1}\| 
	\leq 
	\frac{2^{3/2}(-\xi)^{1/2}}{h}
	\e^{\frac{4}{3h}(-\xi+2^{-3/2}h(-\xi)^{-1/2})^{3/2}} 
	=O(1)(-\xi)^{-1}, 
\end{equation}
when $-\xi\geq h^{2/3}/2$. Here we used that $h(-\xi)^{-3/2} = O(1)$. 
\par 
When $-\xi \leq h^{2/3}/2$ we have $\alpha=h^{2/3}$ and \eqref{eq:REpr5} 
gives 
\begin{equation}\label{eq:REpr12}
	\| P_\xi^{-1}\| 
	\leq 
	\frac{1}{\xi+h^{2/3}}
	\e^{\frac{4}{3h}h} 
	=2\e^{4/3}h^{-2/3}, 
	\quad \text{when } 0< -\xi\leq \frac{h^{2/3}}{2}.
\end{equation}
This concludes the proof of the first part of Proposition \ref{prop1}. 
\end{proof}
\section{Witten Laplacians}
\label{WL}
We are interested in the small singular values of $P_\xi $ and $P_\xi
^*$ when
\begin{equation}\label{WL.1}
  \begin{cases}\xi \ll -h^{2/3},\\
|\xi |-1\ll -h^{2/3}
  \end{cases}
\hbox{for}
\begin{cases}Y={\R},\\
Y=S^1,
\end{cases}
\hbox{respectively.}
\end{equation}
As already mentioned, this amounts to studying the small eigenvalues
of $P_\xi ^*P_\xi $ and $P_\xi P_\xi^*$ respectively.
\par 
Putting
\begin{equation}\label{WL.2}
	f(y,\xi )=y\xi + \phi (y),
\end{equation}
where the parameter $\xi $ will often be suppressed from the notation,
we can write (\ref{eq9}), (\ref{eq9a}) as
\begin{equation}\label{WL.3}
P_\xi =h\partial _y+\partial _yf,\quad P_\xi^* =-h\partial _y+\partial _yf.
\end{equation}
Here we have to keep in mind that $f(\cdot ,\xi )$ is well defined on
${\R}$ but not in general on $S^1$, since $f(y+2\pi ,\xi
)=f(y)+2\pi \xi $, but $\partial _yf$ remains well-defined on $Y$ in
all cases. 
\par 
From (\ref{WL.3}) we see that $P_\xi ^*P_\xi $ and $P_\xi P_\xi
^* $ are semi-classical Schr\"odinger operators
\begin{equation}\label{WL.4}
\begin{split}
&Q_+:=P_\xi ^*P_\xi =-h^2\partial _y^2+(\partial _yf)^2-h\partial
_y^2f,\\ & Q_-:=P_\xi P_\xi^* =-h^2\partial _y^2+(\partial _yf)^2+h\partial
_y^2f.
\end{split}
\end{equation}
To leading order the potential is given by 
\begin{equation}\label{WL.4.5}
	V_0(y,\xi )=(\partial_yf)^2\geq 0.
\end{equation}
The potential $V_0$ vanishes precisely at the critical points of $f$,
given by $\partial _yf(y,\xi )=0$, i.e.\ $\partial _y\phi (y)=-\xi
$. Under the assumption (\ref{WL.1}) we have precisely two critical
points $y_+, y_-\in Y$. In the case $Y={\R}$, they are given by
\begin{equation}\label{WL.5}
y_{\pm}=\pm \sqrt{-\xi },
\end{equation}
and in the case $Y=S^1$, by 
\begin{equation}\label{WL.6}
	\cos(y_\pm) = -\xi. 
\end{equation}
The $\pm$ notation is fixed so that 
\begin{equation}\label{WL.7}
	\pm \partial _y^2f(y_\pm)(y_\pm)=\mp\sin(y_\pm) >0.
\end{equation}
In both cases, we have $\pm \partial _y^2f(y_\pm )>0$. Hence, $f$ is a 
(multi-valued, when $Y=S^1$) Morse function with critical set $\{y_+, y_- \}$, 
a nondegenerate local minimum at $y_+$ and a nondegenerate local maximum at $y=y_-$. 
We notice that there is a degeneration when $\xi \to 0$ in the case 
$Y={\R}$ and when $\xi\to \pm 1$ in the case $Y=S^1$. Similar degeneracies have 
been analyzed by a scaling argument in \cite{Ha06,BM08,Vo14}, and we will follow 
that approach below.
\\
\par 
Recall the definition of the Witten complex,
\begin{equation}\label{WL.8}
hd_f=e^{-f/h}hd\circ e^{f/h}=hd+df^{\wedge}.
\end{equation}
Here $df=\partial _yf\, dy$ is the differential of $f$, viewed as a
differential $1$-form, $dy\in T_y^*Y\simeq {\R}\subset {\C}$ and $df^\wedge$ is
the operator of pointwise left exterior multiplication with $df$;
$df^\wedge (u)=(df)\wedge u$. For $\ell \in {\N}$, let 
$C^\infty (Y;\wedge^\ell {\C})$ denote the space of smooth $\ell$-forms on $Y$. 
Here we work with the usual convention that 
$C^\infty (Y;\wedge^0 {\C}) \simeq C^\infty (Y)$. Since $Y$ is of dimension 
$1$, this space is trivial except for $\ell=0,1$. Thus 
$hd_f:C^\infty (Y;\wedge^\ell {\C})\to C^\infty (Y;\wedge^{\ell +1}
{\C}) $ and the non-trivial case is the one for $\ell=0$:
$hd_f:C^\infty (Y;\wedge^0{\C})\to C^\infty (Y;\wedge^1{\C}).$
\par 
Writing a general 1-form as $u(y)dy$, we get an identification
$C^\infty (Y;\wedge^1{\C})\simeq C^\infty (Y)$ and correspondingly,
$hd_f$ can be identified with the scalar operator
\begin{equation}\label{WL.9}
h\partial _y+\partial _yf = h\partial _y+\xi +\partial _y\phi =P_\xi .
\end{equation}
Viewing $Y$ as a Riemannian manifold with the natural constant metric,
we have 
\begin{equation}\label{WL.10}
hd_f^*=hd^*+df^{\lrcorner}:\, C^\infty (Y;\wedge^{\ell +1}{\C})\to
C^\infty (Y;\wedge^\ell {\C}).
\end{equation}
Here, $d=\partial _y\otimes dy^\wedge$, $d^*=-\partial _y\otimes
dy^{\lrcorner}$ and $dy^\lrcorner$, $df^\lrcorner$ denote the operators of
contraction with $dy$, $\partial _yfdy$, here identified with the
tangent vectors $\partial _y$, $(\partial _yf)\,\partial _y$.
\par 
Again the only non-trivial case is the one where $\ell =0$ and
$hd_f^*$ can be identified with $-h\partial _y+\partial _yf=P_\xi ^*$.
\par 
In the general case of a Riemannian manifold, the Witten
Laplacian is defined as the Hodge Laplacian of the Witten complex,
$$
\Box_f=hd_f^*hd_f+hd_fhd_f^*.
$$
It maps $\ell$-forms to $\ell$-forms and we denote by
$\Box_f^{(\ell)}$ the restriction to $\ell$-forms.
\par 
In our case it is non-trivial only in degree $\ell=0$ and $\ell=1$ and
we have
$$
\Box_f^{(0)}=hd_f^*hd_f,\ \ \Box_f^{(1)}=hd_fhd_f^*.
$$
Again we can identify these two operators with scalar ones,
\begin{equation}\label{WL.11}
\Box_f^{(0)}\simeq P_\xi ^*P_\xi =:Q_+,\ \ \Box_f^{(1)}\simeq P_\xi
P_\xi ^*=:Q_- .
\end{equation}
Our approach follows the one in \cite{HeSj85} in a very much
simplified case and with \cite{HeSj85b} as the main
technical ingredient.
\subsection{Non-degenerate case}
To start with, we restrict $\xi$ to a compact $h$-independent
subset $\mathfrak{K}$ with 
\begin{equation}\label{WL.11.5}
	\xi \in \mathfrak{K} \Subset 
	\begin{cases}
		]-\infty ,0[, \quad Y=\R, \\
		]-1,1[, \quad Y=S^1. 
	\end{cases}
\end{equation}
Recall from \cite{HeSj84} that in fixed neighborhoods of  $y_+$ and
$y_-$ respectively we can construct quasi-modes of the form
\begin{equation}\label{WL.12}
u_k^{\pm}(y;h)=h^{-1/4}a_k^\pm (y;h)e^{-\frac{1}{h}d(y_\pm ,\cdot )},\ 
\end{equation}
$$
a_k^\pm (y;h)\sim \sum_{j=0}^\infty a_{k,j}^\pm (y)h^j\hbox{ in
}C^\infty (\mathrm{neigh\,}(y_\pm)),\ a_{k,0}\not\equiv 0 .
$$
\begin{equation}\label{WL.13}
(P_\xi ^*P_\xi -\mu _k^\pm )u_k^\pm =\mathcal{ O}(h^\infty
)e^{-\frac{1}{h}d(y_\pm ,\cdot )}.
\end{equation}
Here $d(y_\pm ,y)$ denotes the Lithner-Agmon (LA) distance from
$y_\pm$ to $y$, i.e. the distance with respect to the LA metric
$V_0(y)dy^2$, see \eqref{WL.4.5}. For each $k$ we have 
$\mu _k^\pm (h)\sim h(\mu
_{k,0}^\pm+h\mu _{k,1}^\pm + \dots )$ where $\mu _{k,0}^\pm$ is the
$k$:th eigenvalue (with $k=0,1,\dots$) of the quadratic approximation 
of $P_\xi ^*P_\xi $ at $y_\pm$, given by
\begin{equation}\label{WL.14}
(-\partial _y+\partial _y^2f(y_\pm )(y-y_\pm )) (\partial _y+\partial _y^2f(y_\pm )(y-y_\pm )).
\end{equation}
Now, we check that
\begin{equation}\label{WL.15}
d(y_\pm ,y)=\pm (f(y)-f(y_\pm )),\ \ y\in \mathrm{neigh\,}(y_\pm).
\end{equation}
The operator, $P_\xi ^*P_\xi $ is $\ge 0$ and we know that $P_\xi ^*P_\xi
(e^{-(f-f(y_+))/h})=0$ near $y_+$ so up to a $y$-independent prefactor
$u_k^+(y;h)$ is equal to $e^{-(f-f(y_+))/h}$ and we conclude that $\mu
_0^+(h)=0$ in the sense that the asymptotic expansion vanishes
identically.

\par The same discussion applies to $Q_-=P_\xi P_\xi ^*$:
near $y_+$ and $y_-$ we can
construct quasi-modes of the form
\begin{equation}\label{WL.16}
v_k^{\pm}(y;h)=h^{-1/4}b_k^\pm (y;h)e^{-\frac{1}{h}d(y_\pm ,\cdot )},\ 
\end{equation}
such that in a fixed neighborhood of $y_\pm$
\begin{equation}\label{WL.17}
(P_\xi P_\xi^* -\nu _k^\pm )v_k^\pm =\mathcal{ O}(h^\infty
)e^{-\frac{1}{h}d(y_\pm ,\cdot )}.
\end{equation}
For each $k$ we have $\nu _k^\pm (h)\sim h(\nu
_{k,0}^\pm+h\nu _{k,1}^\pm + ... )$ where $\nu _{k,0}^\pm$ is the
$k$:th eigenvalue of the quadratic approximation of $P_\xi P_\xi^* $
at $y_\pm$, given by
\begin{equation}\label{WL.18}
(\partial _y+\partial _y^2f(y_\pm )(y-y_\pm )) 
(-\partial _y+\partial _y^2f(y_\pm )(y-y_\pm )).
\end{equation}
We know that $P_\xi P_\xi^*
(e^{(f-f(y_+))/h})=0$ near $y_-$, so up to a $y$-independent prefactor
$v_k^-(y;h)$ is equal to $e^{(f-f(y_+))/h}$ and we conclude that $\nu
_0^-(h)=0$ in the sense that the asymptotic expansion vanishes
identically.
\par 
Now recall (\ref{WL.4}) and the fact that $\pm\partial _y^2f(y_\pm
)>0$. Comparing the two expressions and using that $\mu _0^+(h)=0$,
$\nu _0^-(h)=0$, we infer that for some constant $C>0,$
\begin{equation}\label{WL.19}
\nu _0^+(h)\ge h/C,\ \ \mu _-^0(h)\ge h/C
\end{equation}
\par 
By standard arguments (see for instance \cite[Theorem 4.23]{DiSj99}) we 
know that if $C>0$ is large enough, then for $h>0$ small
enough, $Q_+$ has a unique eigenvalue $\lambda _+$ in $[0,h/C[$ which
is simple and satisfies $\lambda _+(h)=\mathcal{ O}(h^\infty )$. Similarly
$Q_-$ has a unique eigenvalue $\lambda _-(h)\in [0,h/C[$ which is
simple and $=\mathcal{ O}(h^\infty )$. 
%
%
Let $e_+$, $e_-$ be the
corresponding normalized eigenvectors: $Q_\pm e_\pm =\lambda _\pm
e_\pm$, and let $F_\pm={\C}e_\pm$ denote the associated
1-dimensional eigenspaces. From the intertwining properties,
\begin{equation*}
	P_\xi Q_+=Q_-P_\xi ,\ \ Q_+P_\xi ^*=P_\xi ^*Q_-,
\end{equation*}
we see that
\begin{equation*}
P_\xi :\, F_+\to F_-,\ \ P_\xi ^*:\, F_-\to F_+,
\end{equation*}
and by duality,
\begin{equation*}
P_\xi :\, F_+^\perp\to F_-^\perp,\ \ P_\xi ^*:\, F_-^\perp\to F_+^\perp,
\end{equation*}
If $\lambda _+\ne 0$, we see that $P_\xi e_+$ is an eigenvector of
$Q_-$ with the eigenvalue $\lambda _+$ and hence $\lambda _+=\lambda
_-$. Similarly $\lambda _+=\lambda _-$ if $\lambda _-\ne 0$ and
trivially also if both $\lambda _+$ and $\lambda _-$ vanish, so in
all cases, $\lambda _+=\lambda _-$.
\\
\par 
Let $\pi _\pm:L^2(Y)\to F_\pm$ be the orthogonal projection. By the
results of \cite[Section 2.2]{HeSj85a} we know that 
$K_{\pi _\pm}=\widetilde{\mathcal{O}}(e^{-d(x,y)/h})$ in the sense explained 
in \cite[Defintion 2.2.1]{HeSj85a}. We recall the meaning here for 
the reader's convenience: saying that the operator kernel 
$K_{\pi _\pm}=\widetilde{\mathcal{O}}(e^{-d(x,y)/h})$ 
means that for any $x_0,y_0\in Y$, $\varepsilon>0$ there exist neighborhoods 
$V\subset Y$ of $x_0$, $U\subset Y$ of $y_0$ and a constant $C>0$ such that 
\begin{equation*}
	\| \pi _\pm u \| _{H^1(V)} 
	\leq 
	C_\varepsilon \e^{-(d(x_0,y_0) - \varepsilon)/h}
	\| u \| _{L^2(U)}, 
\end{equation*}
for all $h\in ]0,1]$ and for all $u\in L^2(Y)$ with $\supp u \subset U$. 
\\
\par 
When $Y=S^1$, let $\theta _\pm\in C_0^\infty (]y_\pm-\eta
,y_\pm+\eta [$ be equal to 1 in $[y_\pm-\eta
/2,y_\pm+\eta /2]$. Here $\eta >0$ is a small parameter. Put $\chi
_\pm =1-\theta _\mp$.
\par 
When $Y={\R}$ (cf (\ref{WL.5})), let $\theta _\pm \in C^\infty
({\R})$
\begin{equation*}
\hbox{with }
\begin{cases}\mathrm{supp\,}\theta _+\subset ]y_+-\eta
	,+\infty [,\\
	\mathrm{supp\,}\theta _-\subset ]-\infty ,y_-+\eta [,
\end{cases}
\hbox{ be equal to }1\hbox{ in }
\begin{cases}[y_+-\eta /2,+\infty [,\\
	]-\infty ,y_-+\eta /2]
\end{cases}.
\end{equation*}
Again, put $\chi _\pm =1-\theta _\mp$.
\par 
Let
\begin{equation}\label{WL.19.5}
u_\pm (y;h)=a_\pm (h)\chi _\pm (y)e^{\mp (f(y)-f(y_\pm ))/h}
\end{equation}
where
\begin{equation}\label{WL.19.7}
a_\pm (h)\sim h^{-1/4}(a_0^\pm+a_-^\pm h+...)\hbox{ with }\ a_0>0
\end{equation}
is a normalization factor such that $\| u_\pm\| =1$. Then
\begin{equation}\label{WL.20}
Q_+u_+=a_+[Q_+,\chi _+]e^{-(f-f(y_+))/h}=:r_+,
\end{equation}
\begin{equation}\label{WL.21}
Q_-u_-=a_+[Q_-,\chi _-]e^{(f-f(y_-))/h}=:r_-.
\end{equation}
In particular,
\begin{equation}\label{WL.22}
r_\pm =\widetilde{\mathcal{ O}}_\eta (e^{-S_0/h}),
\end{equation}
where we put $S_0=d(y_+,y_-)$ and let $\widetilde{\mathcal{ O}}_\eta (A)$
denote a quantity which is $\mathcal{ O}(e^{\epsilon (\eta )/h}A)$ where
$\epsilon (\eta )\to 0$ when $\eta \to 0$.
Evaluating $\|u_\pm\|^2$ by the method of stationary phase -- steepest descent
gives the condition 
$$
\frac{\sqrt{2\pi }}{|\partial _y^2f(y_\pm )|^{1/2}}|a_0^\pm |^2=1,
$$
where $\partial _y^2f=\partial _y^2\phi $, so 
\begin{equation}\label{WL.22.2}
a_0^\pm = \left( \frac{|\partial _y^2\phi (y_\pm)|}{2\pi } \right)^{1/4}.
\end{equation}
Let
$$
p_\xi (y,\eta )=i\eta +\xi +\partial _y\phi 
$$
denote the semi-classical principal symbol of $P_\xi $ in (\ref{WL.3}). Then
$$
\frac{1}{2i}\{ p_\xi ,\overline{p}_\xi  \} = \frac{1}{2i}\{i\eta +\xi
+\partial _y\phi , -i\eta +\xi+\partial _y\phi  \}= \partial _y^2\phi 
$$
and (\ref{WL.22.2}) gives
\begin{equation}\label{WL.22.4}
a_0^\pm = \left( \frac{|\{p_\xi ,\overline{p}_\xi  \}(y_\pm,0)|}{4\pi } \right)^{1/4}.
\end{equation}
\par 
We represent $\pi _\pm$ with the Cauchy formula,
\begin{equation}\label{WL.23}
\pi _\pm =\frac{1}{2\pi i}\int _{\partial D(0,h/(2C))}(z-Q_\pm )^{-1}dz,
\end{equation}
where $C>0$ is as in the paragraph after \eqref{WL.19} and 
we know from \cite[Section 2.2]{HeSj85a} that along the 
integration contour,
\begin{equation}\label{WL.24}
K_{(z-Q_\pm)^{-1}}(x,y)=\widetilde{\mathcal{ O}}(e^{-d(x,y)/h}).
\end{equation}
By (\ref{WL.20}),
$$
(Q_+-z)u_+=-zu_++r_+,\hbox{ i.e.\ }(z-Q_+)^{-1}u_+
=\frac{1}{z}u_++(z-Q_+)^{-1}z^{-1}r_+,
$$
and (\ref{WL.23}) gives
\begin{equation}\label{WL.25}
\pi _+u_+=u_++\frac{1}{2\pi i}\int_{\partial
  D(0,h/(2C))}(z-Q_+)^{-1}z^{-1}r_+ dz.
\end{equation}
Similarly,
\begin{equation}\label{WL.26}
\pi _-u_-=u_-+\frac{1}{2\pi i}\int_{\partial
  D(0,h/(2C))}(z-Q_-)^{-1}z^{-1}r_- dz.
\end{equation}
\par 
Let $K\Subset Y$ be compact. Using (\ref{WL.24}) we get 
on $K$ 
\begin{equation}\label{WL.27}
\pi _+u_+-u_+=\widetilde{\mathcal{ O}}_\eta
(e^{-\frac{1}{h}(S_0+d(y_-,\cdot ))})\in F_+^\perp
\end{equation}
\begin{equation}\label{WL.28}
\pi _-u_--u_-=\widetilde{\mathcal{ O}}_\eta
(e^{-\frac{1}{h}(S_0+d(y_+,\cdot ))})\in F_-^\perp .
\end{equation}
In the case when $Y=\R$ we find by the scaling \eqref{eq:LXI1} below 
that for for any $C_0>0$ there exists a $C>0$ such that for $h>0$ small enough 
\begin{equation}\label{WL.28.5}
	\pi _\pm u_\pm-u_\pm=\mathcal{O}(1) \e^{-C_0/h} \text{ on } \R\backslash [-C,C].
\end{equation}
Noticing that \eqref{WL.22} extends to all derivatives of $r_\pm$, we see that the estimates 
in \eqref{WL.27},\eqref{WL.28} and \eqref{WL.28.5} extend to all derivatives of 
$\pi_\pm u_\pm - u_\pm$. 
\par
Using also that $u_\pm =\widetilde{\mathcal{ O}}(e^{-d(y_\pm ,\cdot )/h
})$, we get
\begin{equation}\label{WL.29}
\|\pi _\pm u_\pm\|=1+\widetilde{\mathcal{ O}}_\eta (e^{-2S_0/h}).
\end{equation}
(This also follows from (\ref{WL.27}), (\ref{WL.28}) and the fact that
$u_\pm=\pi _\pm u_\pm+(1-\pi _\pm )u_\pm$ is an orthogonal
decomposition.)
\par 
Let
\begin{equation}\label{WL.30}
e_\pm=\frac{\pi _\pm u_\pm}{\| \pi _\pm u_\pm\|} =\frac{u_\pm}{\|\pi
  _\pm u_\pm\|}+\widetilde{\mathcal{ O}}_\eta
(e^{-\frac{1}{h}(S_0+d(y_\mp ,\cdot ))}),
\end{equation}
where the last term belongs to $F_\pm^{\perp}$, so (\ref{WL.30}) can
be viewed as an orthogonal decomposition in $F_\pm\oplus F_\pm^\perp$
of $u_\pm/\|\pi _\pm u_\pm \|$.
\par 
By construction, $P_\xi u_+=\widetilde{\mathcal{ O}
}_\eta (e^{-(S_0+d(y_- ,\cdot ))/h})$, hence by (\ref{WL.30}),
\begin{equation}\label{WL.31}
P_\xi e_+=\widetilde{\mathcal{ O}
}_\eta (e^{-\frac{1}{h}(S_0+d(y_- ,\cdot ))}).
\end{equation}
More precisely, since
$$
P_\xi (\hbox{the remainder in (\ref{WL.30})})=\widetilde{\mathcal{
    O}}_\eta (e^{-\frac{1}{h}(S_0+d(y_\mp ,\cdot ))}),
$$
we have
\begin{equation}\label{WL.32}
P_\xi e_+=\frac{P_\xi u_+}{\|\pi _+ u_+\|}+\widetilde{\mathcal{ O}}_\eta
(e^{-\frac{1}{h}(S_0+d(y_-,\cdot ))})
\end{equation}
Similarly, $P_\xi ^*u_-=\widetilde{\mathcal{ O}}_\eta
(e^{-(S_0+d(y_+,\cdot ))/h})$,
\begin{equation}\label{WL.33}
P_\xi^* e_-=\frac{P_\xi^* u_-}{\|\pi _- u_-\|}+\widetilde{\mathcal{ O}}_\eta
(e^{-\frac{1}{h}(S_0+d(y_+,\cdot ))}).
\end{equation}
\par 
Recall that $P_\xi :\, F_+^\perp \to F_-^\perp$, $P_\xi ^*:\,
F_-^\perp\to F_+^\perp$, so that the remainders in (\ref{WL.32}),
(\ref{WL.33}) belong to $F_-^\perp$ and $F_+^\perp$ respectively and
we can view these equations as orthogonal decompositions of $P_\xi
u_+/\|\pi _+ u_+\|$, $P_\xi ^*u_-/\|\pi _-u_-\|$ into $F_-\oplus
F_-^\perp$ and $F_+\oplus F_+^\perp$ respectively.
\par 
We have
$$
P_\xi e_+=m _+e_-,\ \ P_\xi ^*e_-=m _-e_+,
$$
for some $m _+,\, m _-\in {\C}$. Since $e_+$, $e_-$ are
normalized, we have
\begin{equation}\label{WL.34}
m _+=(P_\xi e_+|e_-)=(e_+|P_\xi ^*e_-)=\overline{m }_-.
\end{equation}
It follows that the lowest eigenvalue $\lambda _+$ of $Q_+=P_\xi
^*P_\xi $ is given by
\begin{equation}\label{WL.35}
\lambda _+=(P_\xi ^*P_\xi e_+|e_+)=(P_\xi e_+|P_\xi e_+)=|m _+|^2
\end{equation}
and the lowest eigenvalue $\lambda _-$ of $Q_-=P_\xi P_\xi ^*$ is also
equal to $|m _+|^2$ since $\lambda _+=\lambda _-$.
\par 
We take a closer look at $m _+$ in (\ref{WL.34}). We have by 
(\ref{WL.30})
$$
\| \pi _+u_+\|e_+=u_++\widetilde{\mathcal{ O}}_\eta
(e^{-\frac{1}{h}(S_0+d(y_-,\cdot ))})=a_+(h)\chi
_+(y)e^{-\frac{1}{h}d_+(y_+,\cdot )})+\widetilde{\mathcal{ O}}_\eta
(e^{-\frac{1}{h}(S_0+d(y_-,\cdot ))}),
$$
$$
\| \pi _-u_-\|e_-=u_-+\widetilde{\mathcal{ O}}_\eta
(e^{-\frac{1}{h}(S_0+d(y_+,\cdot ))})=a_-(h)\chi
_-(y)e^{-\frac{1}{h}d_-(y_-,\cdot )})+\widetilde{\mathcal{ O}}_\eta
(e^{-\frac{1}{h}(S_0+d(y_+,\cdot ))}),
$$
where $d_\pm$ denotes the LA distance on $Y\setminus  \{y_\mp \}$ when
$Y=S^1$ and $d_\pm = d$ when $Y={\R}$. Recall that these equations
can be viewed as orthogonal decompositions of $u_+$ and $u_-$ and we
get the same for $P_\xi u_+$ and $P_\xi ^* u_-$ after applying $P_\xi
$, $P^*_\xi $ respectively (cf. (\ref{WL.32}), (\ref{WL.33})). It
follows that
\begin{equation*}
(P_\xi u_+|u_-)
	=\| \pi _+u_+\| \| \pi _-u_-\| \underbrace{(P_\xi
 	 e_+|e_-)}_{m _+}+
	 \widetilde{\mathcal{ O}}_\eta
	(e^{-\frac{3}{h}S_0})
\end{equation*}
and we get
\begin{equation}\label{WL.36}
m _+=(P_\xi u_+|u_-)+\widetilde{\mathcal{ O}}_\eta (e^{-\frac{3}{h}S_0}).
\end{equation}
By (\ref{WL.19.5}), (\ref{WL.9}) and the fact that $\chi _-=1$ on the
support of $\partial _y\chi _+$,
\begin{equation}\label{WL.37}
(P_\xi u_+|u_-)=ha_+(h)\overline{a_-(h)}\int
e^{-\frac{1}{h}(d_+(y_+,y)+d_-(y_-,y))}\partial _y\chi _+(y) dy.
\end{equation}
\par 
In the case $Y={\R}$, $y_+$ is to the right of $y_-$ and we get
$$
(P_\xi u_+|u_-)=ha_+(h)\overline{a_-(h)}e^{-S_0/h},
$$
hence
\begin{equation}\label{WL.38}
  m _+=ha_+(h)\overline{a_-(h)}e^{-S_0/h}+\widetilde{\mathcal{
      O}}_\eta (e^{-3S_0/h}).
\end{equation}
\par 
In the case $Y=S^1$, we recall that $\cos y_\pm=-\xi $ with
 $\pm \partial _y^2\phi (y_\pm )=\mp\sin y_\pm >0 $. When $\xi
=-1+\delta $, $0<\delta \ll 1$, we have $y_\pm \approx \mp
\sqrt{2\delta }\, \mathrm{mod\,}2\pi {\Z} $ and when $\xi =1-\delta
$, $0<\delta \ll 1$, we have $y_\pm \approx\pi \pm \sqrt{2\delta }\,
\mathrm{mod\,}2\pi {\Z} $. To fix the ideas, let us consider that
$y_+\in ]-\pi ,0[$, $y_-\in ]0,\pi [$.
\par 
In (\ref{WL.37}) we identify the integration domain with
$J:=]y_--2\pi ,y_-[$. The support of $\partial _y\chi _+$ is contained
in $]y_--2\pi ,y_--2\pi +\eta [\cup ]y_--\eta ,y_-[$. On $]y_--2\pi
,y_--2\pi +\eta [$, $\partial _y\chi _+\ge 0$ and has integral $1$,
while $d_+(y_+,y)+d_-(y_--2\pi ,y)$ is equal to a constant
$d_J(y_+,y_--2\pi )$, the LA distance from $y_+$ to $y_--2\pi $ in
$J$. On $]y_--\eta ,y_-[$, $\partial _y\chi _+\le 0$ and
has integral $-1$, while $d_+(y_+,y)+d_-(y_-,y)$ is a equal to a
constant $d_J(y_+,y_-)$, the LA distance from $y_+$ to $y_-$ in
$J$.

\par Thus (\ref{WL.37}) gives
$$
(P_\xi u_+|u_-)=ha_+(h)\overline{a_-(h)}\left(
  e^{-\frac{1}{h}d_J(y_+,y_--2\pi )}-e^{-\frac{1}{h}d_J(y_+,y_-)} \right)
$$
and hence when $Y=S^1$,
\begin{equation}\label{WL.38.5}
m _+=ha_+(h)\overline{a_-(h)}\left(
  e^{-\frac{1}{h}d_J(y_+,y_--2\pi )}-e^{-\frac{1}{h}d_J(y_+,y_-)}
\right)+\widetilde{\mathcal{ O}}_\eta (e^{-3S_0/h})
\end{equation}
\begin{prop}\label{WL1}
We define $\phi $ as in (\ref{eq6a}) and $P_\xi =h\partial _y+\xi
+\partial _y\phi $ as in (\ref{WL.3}). The common smallest singular value
of $P_\xi $ is equal to the modulus of $m _+$, introduced prior to
(\ref{WL.34}). Restrict $\xi $ to a compact $h$-independent subset of
$]-\infty ,0[$ when $Y={\R}$ and of $]-1,1[$ when $Y=S^1$. Let
$y_+,\, y_-\in Y$ be the two solutions of the equation $\partial
_y\phi (y)=-\xi $, labelled so that $\pm \partial _y^2\phi (y_\pm
)>0$. Let $d$ denote the LA distance on $Y$ for the metric $(\xi
+\partial _y\phi (y))^2dy^2$ and recall (\ref{WL.15}). Let $S_0=d(y_+,y_-)$.
Then uniformly with respect to $\xi $, we have (\ref{WL.38}), when
$Y={\R}$ and (\ref{WL.38.5}) when $Y=S^1$. In the latter case we
identify $y_+$, $y_-$ with points in $]-\pi ,0[$ and $]0,\pi [$
respectively and let $d_J$ denote the $LA$ distance on $J=]y_--2\pi
,y_-[$. Here
$a_+(h)\overline{a_-(h)}=h^{-1/2}(a_0^+\overline{a_0^-}+\mathcal{ O}(h))$ has  a
full asymptotic expansion in powers of $h$ and $a_0^\pm$ are given in
(\ref{WL.19.7}), (\ref{WL.22.4}). In particular,
\begin{equation}\label{WL.38.7}
\begin{split}
m _+=&h^{\frac{1}{2}}\left( \left( \frac{|\{p_\xi ,\overline{p}_\xi
    \}(y_+,0)|}{4\pi } \right)^{1/4}
\left( \frac{|\{p_\xi ,\overline{p}_\xi  \}(y_-,0)|}{4\pi } \right)^{1/4}
+\mathcal{ O}(h)\right)\\ &
\times  \begin{cases}e^{-S_0/h}+\widetilde{\mathcal{
    O}}_\eta (e^{-3S_0/h}),\hbox{ when }Y={\R}\\
\left(e^{-\frac{1}{h}d_J(y_+,y_--2\pi ) }-e^{-\frac{1}{h}d_J(y_+,y_-)} \right) +\widetilde{\mathcal{
    O}}_\eta (e^{-3S_0/h}),\hbox{ when }Y=S^1
\end{cases}
\end{split}
\end{equation}
\end{prop}
Before we continue we comment \eqref{WL.38.7} in the 
case when $Y=S^1$. Notice that $P_0 u_0 = 0$ for $u_0(y) =\exp ( 
-\frac{1}{h}\int_0^y\cos t dt)$, $u_0\in C^\infty(S^1)$. Hence, for $\xi=0$ 
the smallest singular value of $P_\xi$ is $t_0(\xi)=0$. 
\par 
Note that $d_J(y_+,y_--2\pi)$ is a strictly increasing and 
$d_J(y_+,y_-)$ is a strictly decreasing smooth function of $\xi\in ]-1,1[$,  
and their $\partial_\xi$ derivatives are $\asymp 1$. Moreover, 
$d_J(y_+,y_--2\pi) \leq d_J(y_+,y_-)$ for $\xi\in [0,1[$ and 
$d_J(y_+,y_--2\pi) \geq d_J(y_+,y_-)$ for $\xi\in ]-1,1[$ with equality 
in both cases precisely when $\xi =0$. Hence, when $Y=S^1$ and 
$\xi$ is in a compact $h$-independent subset of $]-1,1[$, the smallest singular 
value of $P_\xi$ is given by the modulus of $m_+(\xi)$ with 
\begin{equation}\label{WL.38.9}
\begin{split}
	m _+=&h^{\frac{1}{2}}\left( \left( \frac{|\{p_\xi ,\overline{p}_\xi
		\}(y_+,0)|}{4\pi } \right)^{1/4}
	\left( \frac{|\{p_\xi ,\overline{p}_\xi  \}(y_-,0)|}{4\pi } \right)^{1/4}
	+\mathcal{ O}(h)\right)\\ &
	\times 
	e^{-S_0/h}\left(1-e^{-\frac{\asymp |\xi|}{h}} \right).
\end{split}
\end{equation}
\subsection{The degenerate case -- a dilation argument}\label{sec:SmallSGV_DC}
Now consider the case when
\begin{equation}\label{WL38.8}Y={\R}\hbox{ and } h^{2/3}\ll -\xi \ll 1 \end{equation}
or
\begin{equation}\label{WL38.9}
Y=S^1\hbox{ and }
\begin{cases}\hbox{a) } h^{2/3}\ll 1+\xi \ll 1,\\\hbox{or}\\
\hbox{b) }h^{2/3}\ll 1-\xi \ll 1.
\end{cases}
\end{equation}
The cases a) and b) are similar so when $Y=S^1$, we shall limit the discussion to
the case a).
\par 
When $Y={\R}$, we write $\xi =-\delta $, $h^{2/3}\ll \delta
\ll 1$ and when $Y=S^1$ (case a)) we write $\xi =-1+\delta $, $h^{2/3}\ll \delta
\ll 1$.
\par 
Recall that $P_\xi =h\partial _y+\partial _yf(y,\xi )$, $f(y,\xi
)=y\xi +\phi (y)$, cf. \eqref{WL.2}, \eqref{WL.3}. We have
$$
P_\xi =\begin{cases}h\partial _y+y^2-\delta ,\hbox{ when }Y={\R},\\
h\partial _y+(\cos y-1)+\delta , \hbox{ when }Y=S^1,\hbox{ case a)}.
\end{cases}
$$
We make the dilation $y=\alpha \widetilde{y}$ and keep in mind, when
$Y=S^1$, that $2\pi $-periodic functions of $y$ become $2\pi /\alpha
$-periodic in $\widetilde{y}$. Here $\alpha $ is a small parameter
that we choose $\asymp \delta ^{1/2}$. We get in the two cases,
\begin{equation}\label{WL.39}
P_\xi =\alpha ^2\begin{cases}\widetilde{h}\partial
  _{\widetilde{y}}+\widetilde{y}^2-\widetilde{\delta },\\
\widetilde{h}\partial _{\widetilde{y}}+\frac{\cos (\alpha
  \widetilde{y})-1}{\alpha ^2}+\widetilde{\delta },
\end{cases}, \ \widetilde{y}\in \widetilde{Y}, 
\end{equation}
where $\widetilde{h}=h/\alpha ^3$,
$\widetilde{\delta }=\delta /\alpha ^2 \asymp 1$. Here
\begin{equation}\label{WL.40}
  \widetilde{Y}=\begin{cases}
    Y,\hbox{ when }Y={\R},\\
    {\R}/((2\pi /\alpha ){\Z}), \hbox{ when }Y=S^1,
  \end{cases}
\end{equation}
and $\widetilde{h}=h/\alpha ^3\asymp h/\delta ^{3/2}\ll 1$, since
$\delta \gg h^{2/3}$. The function
$$
\partial _{\widetilde{y}}\widetilde{f}(\widetilde{y})=\begin{cases}
  \widetilde{y}^2-\widetilde{\delta },\\
  \frac{\cos (\alpha \widetilde{y})-1}{\alpha ^2}+\widetilde{\delta
  },
\end{cases} \hbox{ vanishes at }\begin {cases}{\widetilde{y}}_\pm=\pm
  \widetilde{\delta }^{1/2},\\
  \widetilde{y}_\pm\approx\mp (2\widetilde{\delta
  })^{1/2}\,\mathrm{mod\,}(2\pi /\alpha ){\Z},\end{cases}, 
$$
when $Y={\R}$ and $Y=S^1$ respectively. Moreover,
$\pm \partial _{\widetilde{y}}^2\widetilde{f}(\widetilde{y}_\pm )\asymp
1$. We notice that $\partial _{\widetilde{y}}\widetilde{f}$ is uniformly bounded
with respect to $\delta $ on every fixed interval containing
$\{\widetilde{y}_+(\alpha ),\, \widetilde{y}_-(\alpha )\}$ but reaches
large values away from these points. Writing
$P_\xi = \alpha ^2\widetilde{P}$ (cf.\ (\ref{WL.39})) with
$\widetilde{P}=\widetilde{h}\partial _{\widetilde{y}}+\partial
_{\widetilde{y}}\widetilde{f}$, we form
$\widetilde{P}^*=-\widetilde{h}\partial _{\widetilde{y}}+\partial
_{\widetilde{y}}\widetilde{f}$ and
$\widetilde{Q}_+=\widetilde{P}^*\widetilde{P}$,
$\widetilde{Q}_-=\widetilde{P}\widetilde{P}^*$ and repeat the
analysis, now with $\widetilde{h}$ as the small semi-classical
parameter. Let $\widetilde{d}$ denote the LA distance associated to
$\widetilde{Q}_\pm$. Then,
\begin{equation}\label{WL.41}
\widetilde{d}(\widetilde{y}_+,\widetilde{y}_-)\asymp 1.
\end{equation}
In the case when $\widetilde{Y}={\R}/((2\pi /\alpha ){\Z})$, this
distance is obtained by following the uniformly bounded segment from
$\widetilde{y}_+$ to $\widetilde{y}_-$, an interval of length $\approx
2(2\widetilde{\delta } )^{1/2}$ Following the complementary segment
from $\widetilde{y}_+$ to $\widetilde{y}_-$ gives an LA distance which
is $\gg 1$.

\par It follows that the conclusion of Proposition \ref{WL1} extends
to the wider ranges (\ref{WL38.8}), (\ref{WL38.9}) after replacing $h$, $d$, ... by $\widetilde{h}$,
$\widetilde{d}$, ... We review this with a few more details. Recall
that $P=P_\xi $ satisfies
\begin{equation}\label{WL.42}
P=h\partial _y+\partial _yf=\alpha ^2(\widetilde{h}\partial
_{\widetilde{y}}+\partial
_{\widetilde{y}}\widetilde{f}(\widetilde{y}),\ \ \partial _yf=\alpha
^2\partial _{\widetilde{y}}\widetilde{f}
\end{equation}
where $y=\alpha \widetilde{y}$, $\alpha \asymp \delta ^{1/2}$. The LA
metrics in $y$, $\widetilde{y}$ are given by $(\partial _yf)^2dy^2$
and $(\partial _{\widetilde{y}}\widetilde{f})^2d\widetilde{y}^2$
respectively. If $y_j=\alpha \widetilde{y}_j$, $j=1,2$, the
corresponding distances are given by
$$
d(y_1,y_2)=\int_I |\partial _yf||dy|,\ \
\widetilde{d}(\widetilde{y}_1,\widetilde{y}_2)=
\int_{\widetilde{I}} |\partial _{\widetilde{y}}\widetilde{f}||d\widetilde{y}|, 
$$
where $Y\subset Y$, $\widetilde{I}\subset \widetilde{Y}$ are
minimizing segments from $y_1$ to $y_2$ and from $\widetilde{y}_1$ to
$\widetilde{y}_2$ respectively. In the natural sense we can take
$I=\{\alpha \widetilde{y};\, \widetilde{y}\in \widetilde{I} \}$. By
direct computation we have
$$
\widetilde{d}(\widetilde{y}_1,\widetilde{y}_2)=\int_{\widetilde{I}}|\partial
_{\widetilde{y}}\widetilde{f}||d\widetilde{y}|=\int_I \frac{1}{\alpha
  ^2}|\partial _yf|\frac{|dy|}{\alpha }=\frac{1}{\alpha ^3}d(y_1,y_2),
$$
and recalling that $h=\alpha ^3\widetilde{h}$,
\begin{equation}\label{WL.43}
\frac{d(y_1,y_2)}{h}=\frac{\widetilde{d}(\widetilde{y}_1,\widetilde{y}_2)}{\widetilde{h}}.
\end{equation}
\par 
From (\ref{WL.42}) we get
\begin{equation}\label{WL.44}
\partial _y^2f
=\partial _y\alpha ^2 \partial_{\widetilde{y}}\widetilde{f}
=\frac{1}{\alpha }\partial _{\widetilde{y}}\alpha ^2
	 \partial_{\widetilde{y}}\widetilde{f}
=\alpha \partial _{\widetilde{y}}^2\widetilde{f}
\end{equation}
and writing $p(y,\eta )=i\eta +\partial _yf$,
$\widetilde{p}(\widetilde{y},\widetilde{\eta )}=i\widetilde{\eta
}+\partial _{\widetilde{y}}\widetilde{f}$, we get
\begin{equation}\label{WL.45}
  \frac{1}{2i}\{p,\overline{p} \}(y,0)
  = \alpha\frac{1}{2i}
  	\{\widetilde{p},\overline{\widetilde{p}} \}(\widetilde{y},0).
\end{equation}
\par 
We can apply Proposition \ref{WL1} in the $\widetilde{h}$,
$\widetilde{y}$ setting. The smallest singular value of
$\widetilde{P}$ is the only one in an interval 
$[0,\widetilde{h}^{1/2}/C[$. It is of the form $|\widetilde{m }_+|$, where
\begin{equation}\label{WL.46}\begin{split}
\widetilde{m} _+=&\widetilde{h}^{\frac{1}{2}}\left( \left( \frac{|\{\widetilde{p} ,\overline{\widetilde{p}}
    \}(\widetilde{y}_+,0)|}{4\pi } \right)^{1/4}
\left( \frac{|\{\widetilde{p} ,\overline{\widetilde{p}}  \}(\widetilde{y}_-,0)|}{4\pi } \right)^{1/4}
+\mathcal{ O}(\widetilde{h})\right)\\ &
\times  \begin{cases}e^{-\widetilde{S}_0/\widetilde{h}},\hbox{ when }Y={\R}\\
\pm e^{-\widetilde{S}_0/\widetilde{h}}, \hbox{ when }Y=S^1.
\end{cases}
\end{split}
\end{equation}
Here $\widetilde{S}_0$ denotes the $\widetilde{d}$ distance in $\widetilde{Y}$ from
$\widetilde{y}_+$ to $\widetilde{y}_-$. In the case $Y=S^1$ we choose
the $+$ sign when the corresponding minimal segment goes from
$\widetilde{y}_+$ to $\widetilde{y}_-$ in the direction of decreasing
$\widetilde{y}$ values (i.e. in the clock-wise direction) and the $-$
sign when the minimal segment goes from $\widetilde{y}_+$ to
$\widetilde{y}_-$ in the increasing $\widetilde{y}$ direction.

This gives the smallest singular value $|m_+|$ of $P_\xi $, where
$m_+=\alpha ^2\widetilde{m}_+$,
\begin{equation*}
\begin{split}
  m_+=&\alpha ^2\left( \frac{h}{\alpha ^3} \right)^{\frac{1}{2}}
  \left( \left( \frac{|\{p_\xi ,\overline{p}_\xi
    \}(y_+,0)|}{4\pi \alpha  } \right)^{1/4}
\left( \frac{|\{p_\xi ,\overline{p}_\xi  \}(y_-,0)|}{4\pi \alpha  } \right)^{1/4}
+\mathcal{ O}\!\left(\frac{h}{\alpha ^3}\right)\right)\\ &
\times  \begin{cases}e^{-S_0/h},\hbox{ when }Y={\R}\\
\pm e^{-S_0/h},\hbox{ when }Y=S^1
\end{cases},
\end{split}
\end{equation*}
which simplifies to
\begin{equation}\label{WL.47}
\begin{split}
m_+=&h^{\frac{1}{2}}\left( \left( \frac{|\{p_\xi ,\overline{p}_\xi
    \}(y_+,0)|}{4\pi } \right)^{1/4}
\left( \frac{|\{p_\xi ,\overline{p}_\xi  \}(y_-,0)|}{4\pi } \right)^{1/4}
+\alpha ^{\frac{1}{2}}\mathcal{ O}\!\left(\frac{h}{\alpha ^3}\right)\right)\\ &
\times  \begin{cases}e^{-S_0/h},\hbox{ when }Y={\R}\\
\pm e^{-S_0/h},\hbox{ when }Y=S^1
\end{cases}.
\end{split}
\end{equation}
Here we keep in mind that $|\{ p_\xi ,\overline{p}_\xi  \} (y_\pm
,0)|\asymp \alpha \asymp \delta ^{1/2}$, and that $S_0\asymp \delta
^{3/2}$ by (\ref{WL.43}) and the fact that $\widetilde{S}_0 \asymp 1$.
\subsection{The case of $Y=\R$ and large negative $\xi$}
In this section with give a short discussion of the case 
when $Y=\R$ and $-\xi \gg 1$. Writing $\alpha=|\xi|^{1/2}$ 
we see that 
\begin{equation}\label{eq:LXI1}
	P_\xi = \alpha^2 \widetilde{P}_1 
	:= \alpha^2( \widetilde{h}\partial_{\widetilde{y}} +
	(\widetilde{y}^2-1)), 
	\quad \widetilde{h}= \frac{h}{\alpha^{3}}, ~
	y=\alpha\widetilde{y},
\end{equation}
and similarly for $P_\xi^*$. 
\par 
Let $t_0(\xi;h),t_1(\xi;h)$ denote the smallest and second smallest 
singular values of $P_\xi$. The scaling \eqref{eq:LXI1} shows that 
$t_j(\xi;h) = |\xi|t_j(-1;h/|\xi|^{3/2})$. In the limit $|\xi|\gg 1$, 
$h\ll1$, we have $h/|\xi|^{3/2}\ll1$, and we see that for some $C_0,C_1>0$, 
\begin{equation}\label{eq:LXI2}
	t_0(\xi;h) = O(1)|\xi| \frac{h^{1/2}}{|\xi|^{3/4}}\e^{-C_0 |\xi|^{3/2}/h } 
			   = \mathcal{O}(1)|\xi|^{1/4}h^{1/2}\e^{-C_0 |\xi|^{3/2}/h } 
\end{equation}
and 
\begin{equation}\label{eq:LXI3}
	t_1(\xi;h)
	\asymp |\xi|\left(\frac{h}{|\xi|^{3/2}}\right)^{1/2} 
	= |\xi|^{1/4}h^{1/2} 
	\geq h^{1/2}/C_1. 
\end{equation}
Combining Proposition \ref{WL1}, \eqref{WL.47}, and \eqref{eq:LXI2}, \eqref{eq:LXI3} 
yields Theorem \ref{thm:main1}.
\\
\par 
We end this section with collecting some of the above results to give the  
\begin{proof}[Proof of Theorem \ref{thm:main0.1}] 
1. The estimate on $t_1(P_\xi)$ for $\xi\in \mathfrak{K}$ as in \eqref{WL.11.5} 
is an immediate consequence of discussion after \eqref{WL.19}. 
\par
2. For $h^{2/3}\ll -\xi \ll 1$ when $Y=\R$, and for 
$h^{2/3}\ll 1 + \xi \ll 1$ or $h^{2/3}\ll 1 - \xi \ll 1$ when 
$Y=S^1$, the estimate on $t_1(P_\xi)$ follows from the discussion 
after \eqref{WL.45}.
\par 
3. For $-\xi \gg 1$ when $Y=\R$, the estimate on $t_1(P_\xi)$ follows 
from \eqref{eq:LXI3}. 
\end{proof}
\section{Counting singular values}\label{sec:CountingSV}
We shall use \eqref{WL.38.7} and its extension \eqref{WL.47} to study 
the number of singular values $t_0(\xi)=|m_+(\xi)|$ for $\xi \in h\Z$ 
in an interval. We begin by reviewing some properties of $S_0=S_0(\xi)$. 
For $Y=\R$ and $Y=S^1$, respectively, we define the action 
\begin{equation}\label{CSV.0.1}
	S_0:\mathcal{D}(S_0):= \begin{cases}
		]-\infty,0[, \quad Y=\R,\\
		]-1,1[,\quad Y=S^1
	\end{cases}
	\longrightarrow ~~
		]0,+\infty[
\end{equation} 
by 
\begin{equation}\label{CSV.0.2}
	S_0(\xi) = d_\xi(y_-(\xi),y_+(\xi)),
\end{equation} 
given by the LA distance on $Y$ for the metric $(\xi+\partial_y\phi(y))^2dy^2$.
\par
When $Y=\R$, $\xi<0$, we have with $f(y,\xi)=y\xi + \phi(y)=y\xi + y^3/3$ 
that 
\begin{equation*}
\begin{split}
	S_0(\xi)
	&= f(y_-(\xi),\xi) - f(y_+(\xi),\xi) \\
	&= \xi(y_-(\xi)-y_+(\xi)) + \frac{1}{3}(y_-(\xi)^3 - y_+(\xi)^3) \\
	&= \xi(-(-\xi)^{1/2}-(-\xi)^{1/2}) + \frac{1}{3}(-(-\xi)^{3/2} -(-\xi)^{3/2})\\
	&=\frac{4}{3}(-\xi)^{4/3}.
\end{split}
\end{equation*}
Since $y_\pm(\xi)$ are critical points of $y\mapsto f(y,\xi)$, we have 
\begin{equation}\label{CSV.1}
\begin{split}
	\partial_\xi S_0(\xi)
	&= (\partial_\xi f)(y_-,\xi)-(\partial_\xi f)(y_+,\xi)\\
	&= y_-(\xi) - y_+(\xi)  \\
	&= -2(-\xi)^{1/2}.
\end{split}
\end{equation}
\par 
When $Y=S^1$, let $\xi \in ]-1,0[$. (The case $\xi \in]0,1[$ can be 
treated similarly and we shall avoid the case $\xi \approx 0$ to 
avoid complications coming from the two competing terms in \eqref{WL.38.7}). 
In this case, corresponding to the case a) in \eqref{WL38.9}, we have with 
$f(y,\xi)=y\xi+\phi(y)=y\xi+\sin y$, that 
$\partial_\xi S_0(\xi)=y_-(\xi)-y_+(\xi)$, where we recall the convention that 
$y_-\in]0,\pi[$, $y_+\in ]-\pi,0[$, so that $y_\pm\approx 0$, when 
$\xi \approx -1$. In this case, $y_\pm \approx \mp \sqrt{2\delta}$ when 
$\xi =-1+\delta$, $0<\delta \ll 1$. Hence $\partial_\xi S_0(\xi)>0$, 
$\partial_\xi S_0 = (1+\mathcal{O}(\delta))2\sqrt{2\delta}$, when $\xi=-1+\delta$, 
$0<\delta\ll 1$. Since $S_0(-1)=0$, we get 
\begin{equation}\label{CSV.2}
	S_0(\xi)=(1+\mathcal{O}(\delta))\frac{2}{3}(2\delta)^{3/2}.
\end{equation}
\subsection{The non-degenerate case}
We first consider the case when $\xi$ varies in a compact set  
$\mathfrak{K}'$ with 
\begin{equation}\label{CSV.1.5}
	\xi \in \mathfrak{K}' \Subset 
	\begin{cases}
		]-\infty ,0[, \quad Y=\R, \\
		]-1,1[, \quad Y=S^1. 
	\end{cases}
\end{equation}
Then, by Proposition \ref{WL1} and \eqref{WL.38.9}, we have 
for $\xi\in \mathfrak{K}'\backslash\{0\}$
\begin{equation*}
	t_0(\xi) = (1+\mathcal{O}(h))h^{\frac{1}{2}}
		\frac{
			|\{p_\xi,\overline{p}_\xi\}(y_+,0)\{p_\xi,\overline{p}_\xi\}(y_-,0)|^{1/4}
		}{(4\pi)^{1/2}}
		\times 
		\begin{cases}
			\e^{-S_0(\xi)/h}, \\
			\left(1-e^{-\frac{\asymp |\xi|}{h}} \right)\e^{-S_0(\xi)/h},
		\end{cases}
\end{equation*}
where $y_\pm=y_\pm(\xi)$, so
\begin{equation}\label{CSV.3}
	\log\!\left(\frac{t_0(\xi)}{\sqrt{h}}\right) 
	= -\frac{1}{h}S(\xi;h) + \mathcal{O}(h),  
\end{equation}
where 
\begin{equation}\label{CSV.4}
\begin{split}
	S(\xi;h) &= S_0(\xi) 
	-\frac{h}{4}\left(
		\log |\{p_\xi,\overline{p}_\xi\}(y_+,0)|
		+\log |\{p_\xi,\overline{p}_\xi\}(y_-,0)|
	\right)\\ 
	&+ \begin{cases}
		\frac{h}{2}\log 4\pi \\
		\frac{h}{2}\log 4\pi - h\log 
		\left(1-e^{-\frac{\asymp |\xi|}{h}} \right)
	\end{cases}
	\\
	&:= S_0(\xi) + hS_1(\xi).
\end{split}
\end{equation}
Let $K\Subset ]0,+\infty[$ be a fixed compact set and let $0<a<b$ 
with $[a,b]\subset K$. We shall study for $Y=\R$ and $Y=S^1$ 
\begin{equation}\label{CSV.5}
N_{\log (t_0/\sqrt{h})}
\left(\left[-\frac{b}{h},-\frac{a}{h}\right]\right) \\
	:= 
	\# \left\{
		\xi \in h\Z\cap \mathcal{D}(S_0);
		\log\! \left(\frac{t_0(\xi)}{\sqrt{h}}\right) 
		\in \left[-\frac{b}{h},-\frac{a}{h}\right]
	\right\}.
\end{equation}
From the orthogonal decomposition \eqref{eq:int1} we see that the singular 
values of $P$ \eqref{eq1} are of the form $t_k(\xi)$, $k\in\N$, $\xi\in h\Z$, 
where $t_k(\xi)$ denote the singular values of $P_\xi$. By Theorem 
\ref{thm:main0.1} we see that the expression \eqref{CSV.5} yields the 
number of singular values of $P$ in the interval 
\begin{equation}\label{eq:Interval1}
	\left[\sqrt{h}\e^{-b/h},\sqrt{h}\e^{-a/h}\right].
\end{equation}
From \eqref{eq:LXI2}, \eqref{WL.47}, \eqref{WL.38.7} and \eqref{WL.38.9}, 
we deduce that we may restrict $\xi\in h\Z\cap \mathcal{D}(S_0)$ in the 
set on the right hand side of \eqref{CSV.5} to 
$\xi\in h\Z\cap (\mathfrak{K}'\backslash\{0\})$. 
\par 
Since $\xi$ varies in the fixed $h$-indepedendent compact set $\mathfrak{K}'$, 
we find by \eqref{CSV.3} and \eqref{CSV.4} that there exist 
constants $C>0$ (which may vary from line to line) such that 
\begin{equation}\label{CSV.6}
\begin{split}
	\big\{	\xi \in h\Z\cap &(\mathfrak{K}'\backslash\{0\}) ;
	S(\xi;h)\in \left[a+Ch^2,b-Ch^2\right]
	\big\}\\
	&\subset 
	\left\{	\xi \in h\Z\cap (\mathfrak{K}'\backslash\{0\}) ;
	S_0(\xi)\in \left[a+Ch,b-Ch\right]
	\right\} \\
	&\subset 
	\left\{	\xi \in h\Z\cap (\mathfrak{K}'\backslash\{0\}) ;
	\log\! \left(\frac{t_0(\xi)}{\sqrt{h}}\right) 
		\in \left[-\frac{b}{h},-\frac{a}{h}\right]
	\right\} \\ 
	&\subset 
	\left\{	\xi \in h\Z\cap (\mathfrak{K}'\backslash\{0\}) ;
	S(\xi;h)\in \left[a-Ch^2,b+Ch^2\right]
	\right\} \\
	&\subset 
	\left\{	\xi \in h\Z\cap (\mathfrak{K}'\backslash\{0\}) ;
	S_0(\xi)\in \left[a-Ch,b+Ch\right]
	\right\}.
\end{split}
\end{equation}
Defining as in \eqref{CSV.5} 
\begin{equation}\label{CSV.7}
	N_{S_0}(J):= \# \{\xi \in h\Z\cap \mathcal{D}(S_0); S_0(\xi)\in J\}
\end{equation}
and similarly for $S$ we get 
\begin{equation}\label{CSV.8}
	N_{S}([a+Ch^2,b-Ch^2])
	\leq 
	N_{\log (t_0/\sqrt{h})}
	\left(\left[-\frac{b}{h},-\frac{a}{h}\right]\right)
	\leq 
	N_{S}([a-Ch^2,b+Ch^2])
\end{equation}
and
\begin{equation}\label{CSV.9}
	N_{S_0}([a+Ch,b-Ch])
	\leq 
	N_{\log (t_0/\sqrt{h})}
	\left(\left[-\frac{b}{h},-\frac{a}{h}\right]\right)
	\leq 
	N_{S_0}([a-Ch,b+Ch]).
\end{equation}
Sacrificing maximal sharpness for simplicity, we continue with 
\eqref{CSV.9} rather than \eqref{CSV.8}. We see that 
\begin{equation*}
	N_{S_0}([a,a+Ch[),
	N_{S_0}([a-Ch,a[) =O(1), \quad C>0,
\end{equation*}
and similarly with $b$ instead of $a$, so from \eqref{CSV.9}, we get 
\begin{equation}\label{CSV.10}
N_{\log (t_0/\sqrt{h})}
	\left(\left[-\frac{b}{h},-\frac{a}{h}\right]\right)
= 
N_{S_0}([a,b]) + \mathcal{O}(1).
\end{equation}
For $c\in K$, we let $\xi_0(c) \in S_0^{-1}(c)$. Then 
\begin{equation*}
	S_0^{-1}([a,b]) = 
	\begin{cases}
		[\xi_0(b),\xi_0(a)]\\
		[-|\xi_0(a)|,-|\xi_0(b)|] \cup [|\xi_0(b)|,|\xi_0(a)|] ,
	\end{cases}
\end{equation*}
depending on whether $S_0$ is decreasing (in the case of $Y=\R$) or increasing 
on $]-1,0]$ and decreasing on $[0,1[$ (in the case of $Y=S^1$). Approximating the counting measure on $h\Z$ by 
$\frac{1}{h}L(d\xi)$, where $L(d\xi)$ denotes the Lebesgue measure on $\R_\xi$, 
we get
\begin{equation}\label{CSV.11}
	\# \left(
		S_0^{-1}([a,b])\cap h\Z
	\right)
	=
	\frac{1}{h} L(S_0^{-1}([a,b])) +\mathcal{O}(1).
\end{equation}
Combining this with \eqref{CSV.10}, we get 
\begin{equation}\label{CSV.12}
	N_{\log (t_0/\sqrt{h})}
		\left(\left[-\frac{b}{h},-\frac{a}{h}\right]\right)
	= 
	\frac{1}{h} L(S_0^{-1}([a,b])) +\mathcal{O}(1).
\end{equation}
Summing up what we have proven so far we get 
\begin{prop}\label{prop:CSV1}
	Let $S_0$ be as in \eqref{CSV.0.1}, \eqref{CSV.0.2} and let 
	$t_0(\xi)=|m_+(\xi)|$ denote the smallest singular value of $P_\xi$. 
	Let $K\Subset ]0,+\infty[$, when $Y=\R$, and $K\Subset]0,S_0(0)[$, 
	when $Y=S^1$, be a fixed compact set and let $[a,b]\subset K$. Define 
	$N_{\log( t_0/\sqrt{h})}$, $N_{S_0}$ as in \eqref{CSV.5}, \eqref{CSV.7}. 
	Then, 
	\begin{equation*}
		N_{\log (t_0/\sqrt{h})}
			\left(\left[-\frac{b}{h},-\frac{a}{h}\right]\right)
		= 
		N_{S_0}([a,b]) + \mathcal{O}(1).
	\end{equation*}
\end{prop}
\subsection{The degenerate case}
Next, we consider the case when for some fixed large constant $C>0$
\begin{equation}\label{CSV.13}
	\xi \in \begin{cases}
	]-1/C,0[, \quad Y=\R, \\
	]-1,-1+1/C[\cup ]1-1/C,1[, \quad Y=S^1,
	\end{cases}
\end{equation}
and $\xi$ may approach $0$ and $-1$ respectively up to a distance 
$\delta \gg h^{2/3}$. We have seen in \eqref{WL.47} that if 
\begin{equation}\label{CSV.14}	
	\xi = \begin{cases}
	-\delta, \quad Y=\R, \\
	\pm(-1+\delta), \quad Y=S^1,
	\end{cases}
	h^{2/3}\ll \delta \ll 1, 
\end{equation}
then 
\begin{equation*}
	\frac{t_0(\xi)}{\sqrt{h}} = 
	(1+ \mathcal{O}(\delta^{-3/2}h))
	\frac{
	|\{p_\xi,\overline{p}_\xi\}(y_+,0)\{p_\xi,\overline{p}_\xi\}(y_-,0)|^{1/4}
	}{(4\pi)^{1/2}}
	\e^{-S_0(\xi)/h},
\end{equation*}
where $|\{p_\xi,\overline{p}_\xi\}(y_\pm,0)|\asymp \delta^{1/2}$. Hence 
(cf. \eqref{CSV.3}), 
\begin{equation}\label{CSV.15}
	\log\!\left(\frac{t_0(\xi)}{\sqrt{h}}\right) 
	= -\frac{1}{h}S(\xi;h) + \mathcal{O}(h\delta^{-3/2}).
\end{equation}
Here $S$ and $S_1$ are defined in \eqref{CSV.4} and we know that 
$S_1=\mathcal{O}(1)\log \delta^{-1}$, thus 
\begin{equation}\label{CSV.16}
	S= S_0 + \mathcal{O}(h\log \delta^{-1}),
\end{equation}
and \eqref{CSV.15} gives 
\begin{equation}\label{CSV.17}
	\log\!\left(\frac{t_0(\xi)}{\sqrt{h}}\right) 
	= -\frac{1}{h}S_0(\xi) + \mathcal{O}(\log \delta^{-1}) 
		+ \mathcal{O}(h\delta^{-3/2}). 
\end{equation}
We recall that for $\xi$ as in \eqref{CSV.14}
\begin{equation}\label{CSV.18}
	S_0 \asymp \delta^{3/2}, 
	\quad 
	\begin{cases}
		-\partial_\xi S_0 \asymp \sqrt{\delta}, ~Y=\R,\\
		\pm\partial_\xi S_0 \asymp \sqrt{\delta}, ~\xi=\pm(-1+\delta),~Y=S^1.
	\end{cases}
\end{equation}
Let $K\Subset ]0,+\infty[$ be relatively compact in 
$[0,+\infty[$ 
%
%
and fixed. Let $0<a<b$ with $[a,b]\subset K$ with $b\asymp 1$, 
$a\asymp \delta^{3/2}$. Then, 
\begin{equation}\label{CSV.19}
	S_0^{-1}([a,b]) = 
	\begin{cases}
		[\xi_0(b),\xi_0(a)],~ \xi_0(a)\asymp -\delta,~ |\xi_0(b)|\asymp 1,\\
		[-|\xi_0(a)|,-|\xi_0(b)|] \cup [|\xi_0(b)|,|\xi_0(a)|], 
		~1-|\xi_0(a)|\asymp \delta, 
		~1-|\xi_0(b)|\asymp 1.\partial_\xi S_0 \asymp \sqrt{\delta}, ~pm(-1+\delta),~Y=S^1,\\
	\end{cases}
\end{equation}
\par 
Again we shall study $N_{\log (t_0/\sqrt{h})}([-h^{-1}b,-h^{-1}a])$ 
in \eqref{CSV.5}. We now get from \eqref{CSV.17} that there exist 
constants $C>0$ (which may vary from line to line) such that 
\begin{equation}\label{CSV.20}
\begin{split}
&\{\xi \in h\Z\cap \mathcal{D}(S_0)\backslash\{0\};
	S_0(\xi)\in [a+C(h\log\delta^{-1}+h^2\delta^{-3/2}),
	b-Ch]\}
\\
&\subset 
	\left\{	\xi \in h\Z\cap \mathcal{D}(S_0)\backslash\{0\};
		\log\! \left(\frac{t_0(\xi)}{\sqrt{h}}\right) 
			\in \left[-\frac{b}{h},-\frac{a}{h}\right]
	\right\} \\ 
&\subset 
	\{	\xi \in h\Z\cap \mathcal{D}(S_0)\backslash\{0\};
		S_0(\xi)\in [a-C(h\log\delta^{-1}+h^2\delta^{-3/2}),
		b+Ch]
	\}.
\end{split}
\end{equation}
Provided that if $\xi$ belongs to the last set in \eqref{CSV.20}, we have 
$-\xi \gtrsim \delta$ when $Y=\R$ and $1\pm\xi \gtrsim \delta$ when $Y=S^1$. 
By \eqref{CSV.19}, this is the case when $\xi\in S_0^{-1}([a,b])$. If 
$\xi \in S_0^{-1}([a-\mathcal{O}(1)(h\log\delta^{-1}+h^2\delta^{-3/2}),a[)$, 
then using \eqref{CSV.18}, we have $|\xi|-|\xi_0(a)| = \mathcal{O}(1)\delta^{-1/2}
(h\log \delta^{-1}+ h^2\delta^{-3/2})$ and it then suffices to check that 
\begin{equation*}
	\delta^{-1/2}(h\log \delta^{-1}+ h^2\delta^{-3/2}) \ll \delta,
\end{equation*}
i.e. 
\begin{equation*}
	\log \delta^{-1}+ h\delta^{-3/2} \ll h^{-1}\delta^{3/2}.
\end{equation*}
To have this we strengthen the assumption $\delta \gg h^{2/3}$, i.e. 
$\delta^{3/2} \gg h$, to 
\begin{equation}\label{CSV.21}
	h \ll \frac{\delta^{3/2}}{\log \delta^{-1}}.
\end{equation}
If $\xi \in S_0^{-1}(]b,b+\mathcal{O}(h)])$, then $|\xi|-|\xi_0(b)|=\mathcal{O}(h)$ 
and we have the required control as well. 
\par 
We now adopt the assumption \eqref{CSV.21}. It follows from \eqref{CSV.20} 
and the above estimates that 
\begin{equation}\label{CSV.22}
\begin{split}
	\# \bigg\{
	\xi\in h\Z\cap \mathcal{D}(S_0); &
	\log\! \left(\frac{t_0(\xi)}{\sqrt{h}}\right) 
	\in \left[-\frac{b}{h},-\frac{a}{h}\right]
	\bigg\}\\
	&=
	\# \left(S_0^{-1}([a,b])\cap h\Z\right) 
		+\mathcal{O}(1)\frac{\log \delta^{-1}}{\sqrt{\delta}}\\
	&=
	\frac{1}{h} L(S_0^{-1}([a,b]))
	 +\mathcal{O}(1)\frac{\log \delta^{-1}}{\sqrt{\delta}}.
\end{split} 
\end{equation}
Because of the translation invariance of $S_0$ in $x$, we can 
use the measure $\frac{1}{2\pi h}dx\otimes d\xi$ on $S^1\times \R_\xi$ 
in \eqref{CSV.22} and \eqref{CSV.12} instead of just the measure 
$d\xi$ on $\R_\xi$. Furthermore, $dx\otimes d\xi$ can be identified 
with the symplectic volume element $\sigma|_{\Sigma_+}$ on $\Sigma_+$, 
see \eqref{eq:int1.2}, \eqref{eq:int1.3}. Hence, \eqref{CSV.12}, Proposition 
\ref{prop:CSV1} and \eqref{CSV.22} imply Theorem \ref{thm:main2}.
%
%
%
%
%
%
%
\providecommand{\bysame}{\leavevmode\hbox to3em{\hrulefill}\thinspace}
\providecommand{\MR}{\relax\ifhmode\unskip\space\fi MR }
\providecommand{\MRhref}[2]{%
  \href{http://www.ams.org/mathscinet-getitem?mr=#1}{#2}
}
\providecommand{\href}[2]{#2}

\end{document}